\documentclass{article}
\usepackage{array}
\usepackage{amssymb,latexsym,amsthm}
\usepackage{amsmath,amsfonts}
\usepackage{enumerate}

\usepackage{geometry}
\usepackage{url}
\ifdefined\directlua
\usepackage{fontspec}
\usepackage{polyglossia}
\setmainlanguage{english}
\usepackage{microtype}
\usepackage{pdftexcmds}
\makeatletter
\ifcase\pdf@shellescape
\relax\or
\begingroup
  \catcode`\%=12\relax
  \gdef\sformat{"Date: 
\endgroup
\directlua{
 local cmd="git show -s --format='"..\sformat.."'"
 local r=io.popen(cmd):read("*a")
 if (r) then
      tex.print("\string\\def\string\\COMMIT{"..r.."}")
 end
 }
\or
\relax\fi
\makeatother
\ifdefined\COMMIT
        \usepackage{background}
        \backgroundsetup{%
         pages=all, placement=bottom,angle=0,scale=2,%
         vshift=20pt,%
         contents={\COMMIT}}
\fi
\else
\usepackage[british]{babel}
\fi

\newcommand{\KK}{\mathbb{K}}
\newcommand{\RR}{\mathbb{R}}
\newcommand{\SSs}{\mathbb{S}}
\newcommand{\CC}{\mathbb{C}}
\newcommand{\cS}{\mathcal{S}}
\newcommand{\cQ}{\mathcal{Q}}
\newcommand{\cB}{\mathcal B}
\newcommand{\cA}{\mathcal A}
\newcommand{\cU}{\mathcal U}

\newcommand{\wt}{\mathrm{wt}}
\newcommand{\cL}{\mathcal{L}}
\newcommand{\cH}{\mathcal{H}}
\newcommand{\GG}{\mathbb{G}}
\newcommand{\HH}{\mathbb{H}}
\newcommand{\PP}{\mathbb{P}}
\newcommand{\XX}{\mathbb{X}}
\newcommand{\fB}{\mathfrak{B}}
\newcommand{\fX}{\mathfrak{X}}
\newcommand{\fP}{\mathfrak{P}}
\newcommand{\fL}{\mathfrak{L}}
\newcommand{\fA}{\mathfrak{A}}
\newcommand{\fC}{\mathfrak{C}}
\newcommand{\Gr}{\mathrm{Gr}}
\newcommand{\cM}{\mathcal{M}}
\newcommand{\cC}{\mathcal{C}}
\newcommand{\cV}{\mathcal{V}}
\newcommand{\cW}{\mathcal{W}}
\newcommand{\cT}{\mathcal{T}}
\newcommand{\cX}{\mathfrak{X}}
\newcommand{\cF}{\mathcal{F}}
\newcommand{\cP}{\mathcal{P}}

\newcommand{\cN}{\mathcal{N}}
\newcommand{\fa}{\mathfrak a}
\newcommand{\sh}{\mathrm{sh}}
\newcommand{\fb}{\mathfrak b}
\newcommand{\fx}{\mathfrak x}
\newcommand{\fv}{\mathfrak v}
\newcommand{\diam}{\mathrm{diam}}
\newcommand{\PgL}{\mathrm{P}\Gamma{\mathrm{L}}}
\newcommand{\LL}{\mathbb{L}}
\newcommand{\TT}{\mathbb{T}}
\newcommand{\QQ}{\mathbb{Q}}
\newcommand{\cG}{\mathcal{G}}
\newcommand{\ccQ}{\mathcal{Q}}
\newcommand{\RM}{\mathrm{RM}\,}
\newcommand{\trace}{\mbox{\itshape trace}}
\newcommand{\diag}{\mbox{\itshape diag}}
\newcommand{\spin}{\text{\itshape spin}}
\newcommand{\cha}{\text{\itshape char}}
\newcommand{\bE}{\mathbb E}
\newcommand{\baM}{\overline{M}}
\newcommand{\PG}{\mathrm{PG}}
\newcommand{\Sp}{\mathrm{Sp}}
\newcommand{\GL}{\mathrm{GL}}
\newcommand{\PGL}{\mathrm{PGL}}
\newcommand{\PGO}{\mathrm{PGO}}
\newcommand{\PGU}{\mathrm{PGU}}
\newcommand{\PSp}{\mathrm{PSp}}
\newcommand{\FF}{\mathbb{F}}
\newcommand{\ZZ}{\mathbb{Z}}
\newcommand{\NN}{\mathbb{N}}
\newcommand{\WW}{\mathbb{W}}
\newcommand{\bS}{\mathbb{S}}
\newcommand{\Rad}{\mathrm{Rad}}
\newcommand{\Res}{\mathrm{Res}}
\newcommand{\Fix}{\mathrm{Fix}}
\newcommand{\Aut}{\mathrm{Aut}}
\newcommand{\lt}{\mathrm{lt}}
\newcommand{\gr}{\mathrm{gr}}
\newcommand{\er}{\mathrm{er}}
\newcommand{\aut}{\mathrm{Stab}}
\newcommand{\ch}{\mathrm{char}}
\newcommand{\rank}{\mathrm{rank}\,}
\newcommand{\N}{\mathcal{N}}
\newcommand{\ox}{\overline{x}}
\newcommand{\ov}{\overline{v}}
\newcommand{\oy}{\overline{y}}
\newcommand{\oU}{\widetilde{U}}
\newcommand{\oS}{\overline{S}}
\newcommand{\oM}{\overline{M}}
\newcommand{\ou}{\overline{u}}
\newcommand{\oV}{\overline{V}}
\newcommand{\oPi}{{\overline{\Pi}}_{\varphi}}
\newcommand{\oRad}{\overline{\mathrm{Rad}}(\varphi)}
\newcommand{\GF}{\mathrm{GF}}

\newcommand{\bZ}{\bf{0}}
\newcommand{\codim}{\mathrm{codim}\,}
\theoremstyle{plain}
\newtheorem{MainTheo}{Theorem}
\newtheorem{cor}[MainTheo]{Corollary}
\newtheorem{lemma}{Lemma}[section]
\newtheorem{theorem}[lemma]{Theorem}
\newtheorem{corollary}[lemma]{Corollary}
\newtheorem{proposition}[lemma]{Proposition}

\newtheorem{prop}[lemma]{Proposition}

\theoremstyle{definition}
\newtheorem{definition}{Definition}
\newtheorem{remark}[lemma]{Conjecture}
\newtheorem{note}[lemma]{Remark}
\newtheorem{prob}[lemma]{Problem}

\def\pr{\noindent{\bf Proof. }}
\def\eop{\hspace*{\fill}$\Box$}
\begin{document}

\title{The generating rank of a polar Grassmannian}
\author{Ilaria Cardinali, Luca Giuzzi and Antonio Pasini}
\maketitle

\begin{abstract}
In this paper we compute the generating rank of $k$-polar Grassmannians defined over commutative division rings. Among the new results, we compute the generating rank of $k$-Grassmannians  arising from Hermitian forms of Witt index $n$ defined over vector spaces of dimension $N > 2n$.
  We also study generating sets for the $2$-Grassmannians  arising from quadratic forms of Witt index $n$ defined over $V(N,\FF_q)$ for $q=4,8,9$ and $2n \leq N \leq 2n+2$. We prove that for $N >6$ they can be generated over the prime subfield, thus determining their generating rank.
\end{abstract}

\section{Introduction}\label{introduction}

\subsection{Basics on generation and embeddings}

Let $\Gamma=(P,\cL)$ be a point-line geometry, where $P$ is the set of points and $\cL$ is a collection of subsets of $P$, called lines, with the property that any two distinct points belong to at most one line and every line has at least two points. A {\em subspace} of $\Gamma$ is a subset $X \subseteq P$ such that every line containing at least two points of $X$ is entirely contained in $X$. Clearly a subspace $X$ with the lines of $\Gamma$ contained in it is a point-line geometry, which we also denote by $X$.

The intersection of all subspaces of $\Gamma$ containing a given subset $S\subseteq P$ is a subspace called the \emph{span} of $S$
and written as $\langle S\rangle_{\Gamma}.$ We say that $S\subseteq P$ is a \emph{generating set} (or \emph{spanning set}) for $\Gamma$ if $\langle S\rangle_{\Gamma}=P.$ In general,  $\Gamma$ might admit minimal generating sets with different cardinalities. We call the minimum cardinality of a generating set of $\Gamma$ the \emph{generating rank} $\gr(\Gamma)$ of $\Gamma.$

Given a projective space $\PG(V)$ defined over a vector space $V$ and a point-line geometry $\Gamma=(P,\cL)$, a \emph{projective embedding} of $\Gamma$ in $\PG(V)$ (an \emph{embedding}, for short) is an injective map $\varepsilon\colon P\rightarrow \PG(V)$ such that the image of each line of  $\Gamma$ is a projective line and $\langle \varepsilon(P)\rangle_{\PG(V)}=\PG(V).$ The \emph{dimension} of the embedding $\varepsilon$ is the vector dimension of $V$. We write $\varepsilon:\Gamma\rightarrow\PG(V)$ to mean that $\varepsilon$ is a projective embedding of $\Gamma$ in $\PG(V)$.

Given two embeddings $\varepsilon:\Gamma\rightarrow \PG(V)$ and $\varepsilon':\Gamma\rightarrow \PG(V')$ of $\Gamma$ a {\em morphism} from $\varepsilon$ to $\varepsilon'$ is a homomorphism $f:\PG(V)\rightarrow \PG(V')$ (see \cite[chp. 5]{FF}) such that $f\cdot \varepsilon = \varepsilon'$. If $f$ is a collineation then, regarded as a morphism of embeddings, it is called an {\em isomorphism}. If a morphism exists from $\varepsilon$ to $\varepsilon'$ then we say that $\varepsilon'$ is a {\em quotient} of $\varepsilon$ and that $\varepsilon$ {\em covers} $\varepsilon'$.  An embedding is \emph{(absolutely) universal} if it covers all projective embeddings of $\Gamma$.
The universal embedding of $\Gamma$, when it exists, is unique (up to isomorphisms).

The \emph{embedding rank} $\er(\Gamma)$ of $\Gamma$ is the least upper bound of the dimensions of its embeddings. When $\Gamma$ admits the universal embedding then $\er(\Gamma)$ is just the dimension of the latter.  Let $\varepsilon$ and $S$ be a projective embedding and a generating set of $\Gamma$. Then $\varepsilon (\langle S\rangle_{\Gamma})\subseteq \langle \varepsilon (S)\rangle_{\PG(V)}$. Hence $\dim(\varepsilon)\leq \gr(\Gamma)$.  So, $\gr(\Gamma) \geq \er(\Gamma)$. In particular, suppose that $\Gamma$ admits the universal embedding and $\gr(\Gamma) < \infty$. Suppose moreover that $\dim(\varepsilon)=\gr(\Gamma)$ for a particular embedding $\varepsilon$ of $\Gamma$. Then $\varepsilon$ is universal.

\subsection{The problem studied in this paper}

Let $\cal P$ be a non-degenerate Hermitian or orthogonal polar space of finite rank $n \geq 2$. In this paper we investigate the generating rank $\gr({\cal P}_k)$ of the $k$-Grassmannian ${\cal P}_k$ of $\cal P$ for $1 \leq k \leq n$. A few results on the embedding rank $\er({\cal P}_k)$ will be obtained as a by-product.

We refer to Section~\ref{Prelim} for the terminology, in particular for the definition of the defect of a non-degenerate polar space. Provisionally, we define the defect of $\cal P$ as the difference $d := \er({\cal P})-2n$. Note that $d$ can be infinite, but if the underlying field of $\cal P$ is finite then $d \leq 1$ if $\cal P$ is Hermitian and $d\leq 2$ if it is  quadratic.

Symplectic polar spaces will not be considered in this paper. Indeed if $\cal S$ is a non-degenerate symplectic polar space of rank $n$ in characteristic different from $2$ then the embeddings of its $k$-Grassmannians ${\cal S}_k$ are well understood and $\gr({\cal S}_k)$ is known, for every $k = 1, 2,..., n$. Explicitly, $\gr({\cal S}_k) = \er({\cal S}_k) = {{2n}\choose k}-{{2n}\choose{k-2}}$; see~\cite{BB98,B2007,BC2012,C98,CS01,B09,BP07,Premet}. We have nothing to add to this. On the other hand, symplectic polar spaces in characteristic $2$ should rather be regarded as orthogonal spaces (see e.g. \cite{BP09}); we adopt this point of view in this paper.

\subsubsection{A survey of known results}

Hereby we summarize all results on generating and embedding ranks of Hermitian or orthogonal Grassmannians we have found in the literature. Throughout this subsection and the following one $\FF$ always stands for the underlying field of the polar space under consideration.

Let $\cH$ be a non-degenerate Hermitian polar space of finite rank $n$ with defect $d = 0$ and let ${\cal H}_k$ be its $k$-Grassmannian for $1\leq k \leq n$. When $k > 1$ suppose that $\FF\not=\FF_4$. Then $\gr(\cH_k)= \er({\cal H}_k) = {2n\choose k}$; see \cite{BB98, BC2012, CS01, B10, BP07}.

Let $\FF = \FF_4$ and $k = n$. In this case $\er({\cal H}_n) = (4^n+2)/3$ (Li \cite{Li02}); most likely $\gr({\cal H}_n) = \er({\cal H}_n)$ for any $n \geq 2$ but a proof of this equality is known only for $n \leq 3$ (see Cooperstein \cite{Co2001} for $n = 3$; the case $n = 2$ is very well known).

We have found nothing in the literature on the generating rank of $\cH_k$ when $d > 0$ but the following result \cite{CS01}: if $d =1$ and $\FF$ is finite then $\gr(\cH_n)\leq 2^n$.

The orthogonal case is harder. Let $\cQ$ be a non-degenerate orthogonal polar space of rank $n$ and defect $d \leq 2$ and let ${\cal Q}_k$ be the $k$-grassmannian of $\cal Q$ for $1\leq k \leq n$. Nothing is known on $\gr({\cal Q}_k)$ for $2 < k < n$ but lower bounds provided by known embeddings of ${\cal Q}_k$. It is folklore that $\gr({\cal Q}_1) = \er({\cal Q}_1) = 2n+d$; anyway, see Corollary \ref{t:k=1 symplectic even} of this paper.

Let $k = n$. If $d = 0$ then $\cal Q$ is a hyperbolic quadric. This case is devoid of interest. Indeed in this case all lines of ${\cal Q}_n$ have just two points; consequently ${\cal Q}_n$ is not embeddable and the full set of points of ${\cal Q}_n$ is the unique generating set of ${\cal Q}_n$. So, let $d > 0$. When $1 \leq d \leq 2$ and $\mathrm{char}(\FF) \neq 2$ then $\cQ_n$ is embeddable and $\gr({\cal Q}_n) = \er({\cal Q}_n) = 2^n$ (see \cite{W82}, \cite{BB98}, \cite{CS01}, \cite{B11}).

Let $\ch(\FF)=2$ and $d > 0$. Then in general $\cQ$ admits many embeddings, but all of them cover a unique embedding, the minimum one, which can be obtained by factorizing the natural (universal) embedding of $\cQ$ over the radical $R$ of the bilinearization of the quadratic form $q$ associated to $\cQ$. If $R = 0$ and $d = 2$ then $\cQ$ admits a unique embedding. In this case $\cQ_n$ is embeddable and we have $\gr({\cal Q}_n) = \er({\cal Q}_n) = 2^n$, just as when $\ch(\FF) \neq 2$ (see \cite{CS01}, \cite{B11}). Suppose that $R \neq 0$ (which is always the case when $d = 1$) and that
$q(R) = \FF$ (as it is the case when $\FF$ is perfect and, consequently, $d = 1$). Then the minimum embedding embeds $\cQ$ as a symplectic space in  $\PG(2n-1,\FF)$ (see \cite{BP09, Pas17}). In this case $Q_n$ is embeddable and, if $\FF\neq \FF_2$, then $\gr(\cQ_n)= \er({\cal Q}_n) = {2n\choose n}-{2n\choose {n-2}}$ (see \cite{BP07}). If $\FF=\FF_2$ then $\er({\cal Q}_n) = (2^n+1)(2^{n-1}+1)/2$ (Li \cite{Li01}); it is likely that $\gr({\cal Q}_n) = \er({\cal Q}_n)$ for $\FF = \FF_2$ too, but this equality has been proved only for $n\leq 5$; see~\cite{Co1997}.

Finally, let $k = 2 < n$. Assume that $\FF$ is a finite prime field. Then $\gr({\cal Q}_2) = \er({\cal Q}_2) = {{2n+d}\choose 2}$ (see \cite{C98b, BPa01} for $d = 0$, \cite{BB98, CS97} for $d = 1$ and \cite{C98b} for $d = 2$). Nothing is known on $\gr({\cal Q}_2)$ for $\FF$ an arbitrary field except that if $d \leq 1$ then $\gr({\cal Q}_2) \leq {{2n+d}\choose 2}+g$ where $g$ is the minimal number of elements to be added to the prime subfield of $\FF$ in order to generate the whole of $\FF$ (see \cite{BPa01}).

Some more is known on $\er({\cal Q}_k)$. Explicitly, let $k \in \{2, 3\}$, $k < n$ and $d \leq 1$. Suppose that  $\FF$ is a perfect field of positive characteristic or a number field. When $k = 3$ suppose furthermore that $\FF \neq \FF_2$. Then $\er({\cal Q}_k) = {{2n+d}\choose k}$ (see \cite{IP13}).

\subsubsection{The main results of this paper}\label{Main Sec}

Let $\cH$ be a non-degenerate Hermitian polar space of finite rank $n$ and defect $d$ and let $\cH_k$ be the $k$-Grassmannian of $\cH$, for $1 \leq k \leq n$. Let $\FF$ be the underlying field of $\cH$.

\begin{MainTheo}\label{MT1}
Both the following hold:
\begin{enumerate}
\item Let $k<n$ and $d < \infty$. When $1 < k$ suppose moreover that $\FF$ has order $|\FF| > 4$. Then $\gr(\cH_k)={N\choose k}$, where $N=2n+d$.
\item Let $k=n$. If $d > 0$ then $\gr(\cH_n)\leq 2^n$.
\end{enumerate}
\end{MainTheo}
\noindent
By \cite{ILP}, for $k<n$ the geometry $\cH_k$ affords a projective embedding $\varepsilon_k$ in $\PG(\bigwedge^kV)$ called the \emph{Pl\"ucker embedding}, where $V = V(N,\FF)$ is the vector space hosting the (unique) embedding of $\cal H$. We have $\dim(\varepsilon_k)= \mathrm{dim}(\bigwedge^kV) = {{N}\choose {k}}$.
It follows from \cite{K-S} that ${\cal H}_k$ admits the universal embedding. By the first part of Theorem \ref{MT1} we get the following:

\begin{cor}
If $d < \infty$, $k<n$ (and $\FF \neq \FF_4$ when $k > 1$) then the Pl\"ucker embedding of $\cH_k$ is universal (whence $\er({\cal H}_k) = \gr({\cal H}_k)$).
\end{cor}

Turning to the case of $k = n$, it is likely  that if $d > 0$ then $\gr(\cH_n) = 2^n$. However when $d > 0$ the geometry $\cH_n$ admits no projective embedding. Consequently, there is no easy way to replace the inequality $\gr(\cH_n)\leq 2^n$ with the corresponding equality.

Let now $\cQ$ be a non-degenerate orthogonal polar space of finite rank $n$ and defect $d$. For $1\leq k \leq n$ let $\cQ_k$ be its $k$-Grassmannian. Let $\FF$ be the underlying field of $\cQ$.

\begin{MainTheo}\label{MT2}
Let $k=n$. If $d > 0$ and $\ch(\FF)\neq 2$ then $\gr(\cQ_n)\leq 2^n.$
\end{MainTheo}

As remarked in the previous subsection, when $0 < d \leq 2$ and $\ch(\FF) \neq 2$ the inequality $\gr(\cQ_n)\leq 2^n$ is in fact an equality. Perhaps the same is true when $d > 2$ but, since $\cQ_n$ is not embeddable when $d > 2$, there is no easy way to prove it.

In order to state our next results we need to fix some notation an terminology. Let $V = V(N,\FF)$ and $q:V\rightarrow\FF$ be a quadratic form. Let $\FF_0$ be a proper subfield of $\FF$ and $B$ a basis of $V$ such that the polynomial which represents $q$ with respect to $B$ is defined over $\FF_0$. Let $V_0$ be the $N$-dimensional $\FF_0$-vector space of the $\FF_0$-linear combinations of the vectors of $B$. Consider the form $q_0\colon V_0\rightarrow \FF_0$ induced by $q$ on  $V_0$ and let $\cQ(\FF)$ and $\cQ(\FF_0)$ be the polar spaces associated to $q$ and $q_0$ respectively. Regarded $\PG(V_0)$ as a subgeometry of $\PG(V)$ as usual and recalling that $\cQ(\FF_0)$ and $\cQ(\FF)$ are subgeometries of $\PG(V_0)$ and $\PG(V)$ respectively, the polar space $\cQ(\FF_0)$ is a subgeometry of $\cQ(\FF)$. So, we can consider its span $\langle\cQ(\FF_0)\rangle_{\cQ(\FF)}$ in $\cQ(\FF)$. If $\langle\cQ(\FF_0)\rangle_{\cQ(\FF)}=\cQ(\FF)$ then we say that $\cQ(\FF)$ is $\FF_0$-\emph{generated}.

The above naturally extends to $k$-Grassmannians of $\cQ(\FF)$. Explicitly, every $k$-subspace $X$ of $\cQ(\FF_0)$, regarded as a subspace of $V_0$, whence a subset of $V$, spans in $V$ a $k$-subspace $\FF\otimes X = \langle X\rangle_V$ of $\cQ(\FF)$. So, the $k$-Grassmannian $\cQ_k(\FF_0)$ of $\cQ(\FF_0)$ can be naturally embedded into the
$k$-Grassmannian $\cQ_k(\FF)$ of $\cQ(\FF)$ by means of the map
\[\begin{array}{rcccc}
\xi & \colon & \cQ_k(\FF_0) & \to & \cQ_k(\FF)\\
 & & X & \to & \langle X\rangle_V
\end{array}\]
\begin{definition} \label{subgeometry}
If $\cQ_k(\FF)=\langle\xi(\cQ_k(\FF_0))\rangle_{\cQ_k(\FF)}$ then we say that $\cQ_k(\FF)$ is {\em generated over} $\FF_0$ (also $\FF_0$-{\em generated}, for short).
\end{definition}

Clearly, if $\cQ_k(\FF)$ is $\FF_0$-generated then $\gr(\cQ_k(\FF)) \leq \gr(\cQ_k(\FF_0))$. So, if we already know $\gr(\cQ_k(\FF_0))$ and we also know that $\cQ_k(\FF)$ admits an embedding of dimension equal to $\gr(\cQ_k(\FF_0))$, then we can conclude that $\er(\cQ_k(\FF)) = \gr(\cQ_k(\FF)) = \gr(\cQ_k(\FF_0))$.

As recalled in the previous subsection, if $n > 2$, $d \leq 1$ and $\FF_0$ is a prime field then $\gr(\cQ_2(\FF_0)) = {{2n+d}\choose 2}$ while, if $\FF$ is any field admitting $\FF_0$ as its prime subfield, then $\gr(\cQ_2(\FF)) \leq \gr(\cQ_2(\FF_0))+g$ where $g$ is the minimal size of a subset $G\subset \FF$ such that $\FF_0\cup G$ generates $\FF$. On the other hand, $\cQ_2(\FF)$ admits the Weyl embedding, which is just ${{2n+d}\choose 2}$-dimensional \cite{IP13}. So, it is natural to ask if $\cQ_2(\FF)$ is $\FF_0$-generated.

This problem is studied in \cite{BPa01} but the results obtained in that paper fall short of a complete solution. The following two theorems deal indeed with this problem. Admittedly, they don't contribute so much to its solution, nevertheless they look intriguing. They seem to point at opposite directions: Theorem \ref{MT3} apparently suggests a negative answer to our problem while Theorem \ref{MT4}, where a few particular cases are considered, lets us hope for an affirmative one.

\begin{MainTheo}\label{MT3}
Let $\cQ(\FF)$ be a non-degenerate hyperbolic polar space of rank $n = 3$ in $\PG(V)$, where $V = V(6,\FF)$ (so, $\cQ(\FF)$ has defect $d = 0$). Suppose that $\FF$ is not a prime field. Then the line-Grassmannian $\cQ_2(\FF)$ of $\cQ(\FF)$ is never $\FF_0$-generated, for any proper subfield $\FF_0$ of $\FF$.
\end{MainTheo}

\begin{MainTheo}\label{MT4}
Let $\cQ(\FF)$ be a non-degenerate orthogonal polar space of rank $n\geq 2$ and defect $d \leq 2$, defined in $\PG(V)$ where $V = V(N,\FF)$ with $N=2n+d$, and let $\cQ_2(\FF)$ be its line-Grassmannian. Suppose that $N > 6$ and $\FF = \FF_q$ with $q\in \{4, 8, 9\}$. Then $\gr(\cQ_2(\FF))= \er(\cQ_2(\FF)) = {{N\choose 2}}$. Moreover, if $d\leq 1$ then $\cQ_2(\FF)$ is generated over the prime subfield of $\FF$.
\end{MainTheo}
\noindent
So far we have only considered the cases $k = n$ and $k = 2$. We shall prove later (Corollary \ref{t:k=1 symplectic even}) that $\gr(\cQ_1) = \er(\cQ_1) = 2n+d$ for any choice of $n$ and $d$, but this is not surprising at all. Regretfully, we have no sharp result to offer related to cases where $2 < k < n$.

\paragraph{Organization of the paper}
The paper is organized as follows. Section~\ref{Prelim} is dedicated to definitions and basic results on polar spaces; in particular, we introduce the defect of a non-degenerate polar space and we study some of its properties. In Section \ref{sec3new} we prove some general results on generating sets of polar Grassmannians. The results of Section 3 will be exploited in Sections \ref{Herm} and \ref{OG}, where Hermitian Grassmannians and orthogonal Grassmannians are studied.

\section{Preliminaries on polar spaces}\label{Prelim}

\subsection{Polar spaces, their Grassmannians and subspaces}\label{basics}

Let ${\cal P} = (P, {\cal L})$ be a non-degenerate polar space of finite rank $n \geq 2$ with no thin lines, regarded as a point-line geometry. Following \cite{BuekC}, we denote the collinearity relation of $\cal P$ by the symbol $\perp$. For every point $x\in P$, we denote by $x^\perp$ the set of points of $\cal P$ collinear with $x$, including $x$ among them; given a subset $X\subseteq P$, we put $X^\perp := \bigcap_{x\in X}x^\perp$.

All subspaces of $\cal P$ are possibly degenerate polar spaces of rank at most $n$. In particular, for a subspace $X$ of $\cal P$ it can be that $X\subseteq X^\perp$. If this is the case then $X$ is called a {\em singular} subspace. All singular subspaces of $\cal P$ are projective spaces (see \cite{BuekC}) of rank at most $n$ (recall that the rank of a projective space is its dimension augmented by 1), those of rank $n$ being the maximal ones. Henceforth, if $X$ is a singular subspace of $\cal P$ of rank $k$ we say that $X$ is a {\em singular} $k$-{\em subspace}, also a $k$-{\em subspace} for short. Clearly, the $1$-subspaces and $2$-subspaces are just the points and the lines of $\cal P$. If $n > 2$ the $3$-subspaces are also called {\em planes}; we allow $\emptyset$ as the unique $0$-subspace. For $k = 0, 1,..., n$ we denote the set of $k$-subspaces of $\cal P$ by the symbol $S_k({\cal P})$.

Given a $k$-subspace $X$ of $\cal P$ with $k < n$, the {\em upper residue} $\Res(X)^\uparrow$ of $X$ (also called the {\em star} of $X$) is the collection of all singular subspaces of $\cal P$ properly containing $X$. The upper residue of $X$ naturally yields a polar space of rank $n-k$, where $\Res(X)^\uparrow\cap S_{k+1}({\cal P})$ is the set of points and, if $k < n-1$, the $(k+2)$-subspaces containing $X$ play the role of lines, the set $\Res(X)^\uparrow$ being naturally identified with the family of (nonempty) singular subspaces of this polar space.

Conversely, let $X$ be a $k$-subspace with $k > 1$. The {\em lower residue} $\Res(X)^\downarrow$ of $X$ is the collection of all singular subspaces of $\cal P$ properly contained in $X$, namely the family of all proper subspaces of $X$, the latter being regarded as a projective space.

When $X$ has rank $n$ its upper residue is not defined; so we feel free to write $\Res(X)$ instead of $\Res(X)^\downarrow$, calling $\Res(X)$ the {\em residue} of $X$. Similarly, if $x$ is a point we write $\Res(x)$ instead of $\Res(x)^\uparrow$ and call $\Res(x)$ the {\em residue} of $x$.

We are now ready to define polar Grassmannians. Let $1\leq k < n$. Then the $k$-Grassmannian ${\cal P}_k$ of $\cal P$ is the point-line geometry with $S_k({\cal P})$ as the set of points and the following subsets of $S_k({\cal P})$ taken as lines:
\[\Res(X)^\uparrow\cap \Res(Y)^\downarrow~\subseteq~ S_k({\cal P}), \hspace{5 mm} \mbox{for}~X\in S_{k-1}({\cal P}),~ Y\in S_{k+1}({\cal P}), ~ X \subset Y.\]
Note that this definition also makes sense when $k = 1$, since $S_0({\cal P}) = \{\emptyset\}$ and $\emptyset$ is a singular subspace of $\cal P$, by convention. Clearly, ${\cal P}_1$ is just the same as $\cal P$.

On the other hand, let $k = n$. The set of points of the $n$-Grassmannian ${\cal P}_n$ is $S_n({\cal P})$ and the lines are the upper residues of the members of $S_{n-1}({\cal P})$. The geometry ${\cal P}_n$ is usually called a {\em dual polar space}.

Singular subspaces are the only subspaces of $\cal P$ we have considered so far, but in this paper we will often deal with other kinds of subspaces, as hyperplanes and family of subspaces, which we like to call nice.

A {\em hyperplane} is a proper subspace $H$ of $\cal P$ meeting every line of $\cal P$ non-trivially. In particular, for every point $x\in P$ the set $x^\perp$ is a hyperplane, usually called a {\em singular hyperplane}, with $x$ as its {\em deep point}. Needless to say, singular hyperplanes are not singular subspaces, in spite of the word `singular' used to name them. Indeed the singular hyperplanes are precisely the hyperplanes which, regarded as polar spaces, are degenerate, their deep points being their radicals. The quotient $H/x$ of $H = x^\perp$ over its radical $x$ is just the residue $\Res(x)$.

It is well known \cite{CS90, Pas90} that the collinearity relation $\perp$ induces a connected graph on the complement $P\setminus H$ of a hyperplane $H$.
This fact, combined with \cite[Lemma 4.1.1]{SH11}, implies that every hyperplane is a maximal subspace.

Let $x, y$ be two non-collinear points of $\cal P$. Then $\{x,y\}^\perp$ is a non-degenerate subspace of $\cal P$ isomorphic to $\Res(x)$ ($\cong \Res(y)$).
The double perp $\{x,y\}^{\perp\perp}$ is called a {\em hyperbolic line}. It contains $x$ and $y$ and no two of its points are collinear.

Finally, a {\em nice subspace} of $\cal P$ is a subspace containing two mutually disjoint maximal singular subspaces of $\cal P$. In other words, a subspace of $\cal P$ is nice if $\cal P$ induces on it a non-degenerate polar space of the same rank $n$ as $\cal P$. Let $\mathfrak{N}({\cal P})$ be the family of nice subspaces of $\cal P$, ordered by inclusion. Clearly $\cal P$ is the greatest element of $\mathfrak{N}({\cal P})$. The minimal elements of $\mathfrak{N}({\cal P})$ are the subspaces spanned by the unions of pairs of mutually disjoint maximal singular subspaces.

We recall that, given two disjoint maximal singular subspaces $M$ and $M'$ of $\cal P$ and basis $\{p_1,..., p_n\}$ of $M$, a unique basis $\{p',..., p'_n\}$ exists in $M'$ such that $p_i\perp p'_j$ if and only if $i\neq j$. A pair $\{\{p_1,..., p_n\},\{p'_1,..., p'_n\}\}$ of bases as above is called a {\em frame} of $\cal P$ (see \cite{PasDG},~\cite{BuekC}). The minimal elements of $\mathfrak{N}({\cal P})$ are just the subspaces of $\cal P$ that can be generated by (the points of the two bases of) a frame.

Recall that the length of a chain is its cardinality diminished by $1$, with the usual convention that $\mathfrak{n}-1 = \mathfrak{n}$ when $\mathfrak{n}$ is infinite.

\begin{definition}\label{defect}
The {\em anisotropic defect} $\mathrm{def}({\cal P})$ of $\cal P$ (the {\em defect} of $\cal P$ for short) is the least upper bound for the lengths of the well ordered chains of $\mathfrak{N}(\cP)$.
\end{definition}

This definition is never vacuous, since $\mathfrak{N}(\cP)$ always contains finite chains and, trivially, every finite chain is well ordered. In general, not all chains of $\mathfrak{N}(\cP)$ are well ordered and non-well ordered chains exist larger than $\mathrm{def}(\cP)$ (see Remark \ref{monster chains}). However, when $\mathfrak{N}(\cP)$ has finite length, namely all of its chains are finite,
then all chains of $\mathfrak{N}(\cP)$ are well ordered. In this case the above definition admits simpler formulations (see e.g. Theorem \ref{3}).
In Subsection \ref{anisotropic defect} we shall prove that, when $\cal P$ is embeddable and defined over a field, then $\mathrm{def}({\cal P}) = \er({\cal P})-2n$.

\subsection{Polar spaces defined over fields}\label{underlying field}

When $\rank(\cP) = n > 3$ all maximal singular subspaces of $\cP$ are isomorphic to $\PG(n-1,\KK)$ for a given division ring $\KK$ (Tits \cite[chapter 7]{Tits}). This fails to hold in general when $n \leq 3$. However, keeping the hypothesis $n \geq 2$, suppose that $\cP$ is embeddable, but not a grid when $n = 2$. Then all embeddings of $\cP$ are defined over the same division ring, which is taken as the {\em underlying division ring} of $\cP$. Moreover $\cP$ admits the universal embedding, except in one exceptional case of rank $2$ defined over a non-commutative division ring (Tits \cite[\S 8.6]{Tits}).

Suppose that the underlying division ring of $\cP$ is a field, say $\FF$, and let $\varepsilon:{\cP}\rightarrow \PG(V)$ be the universal embedding of $\cP$, where $V = V(N,\FF)$ and $N = \er(\cP)$. So, $\varepsilon(\cP)$ is the polar space $\cP(f)$ associated to a non-degenerate alternating, Hermitian or quadratic form $f$ of $V$, namely the singular subspaces of $\cP$ are the subspaces of $\PG(V)$ corresponding to subspaces of $V$ totally singular for $f$.

Note that, since $\varepsilon$ is universal, $f$ cannot be alternating when $\ch(\FF) = 2$ and, if $\ch(\FF)\neq 2$ or $\ch(\FF) = 2$ but $f$ is Hermitian, then $\varepsilon$ is the unique embedding of $\cP$ (see Tits \cite[Chapter 8]{Tits}). On the other hand, let $f$ be quadratic and $\ch(\FF) = 2$. Then $\varepsilon$ in general admits several proper quotients. All of them can be described by means of generalized quadratic forms as defined in \cite{Pas17} with  at most one exception, which occurs when the minimum quotient is associated to an alternating form.

We recall the definition of generalized quadratic form here, since we will need it in the next subsection. Let $\FF^2 := \{t^2\}_{t\in \FF}$. This is a subfield of $\FF$, since we are assuming that $\ch(\FF) = 2$. Let $U$ be a proper subgroup of the additive group of $\FF$ with the property that $U\FF^2 \subseteq U$. The quotient group $\FF/U$ can be regarded as an $\FF^2$-vector space, with scalar multiplication defined as follows: $(t+U)\cdot\lambda^2 = t\lambda^2 + U$ for $t, \lambda\in \FF$. A mapping $\phi:V'\rightarrow \FF/U$ from an $\FF$-vector space $V'$ to $\FF/U$  is said to be a {\em generalized quadratic form} if it satisfies the same properties which characterize quadratic forms except that its values belong to $\FF/U$. Explicitly,
\[\phi(x\lambda + y\mu) =  \phi(x)\cdot\lambda^2 + \phi(y)\cdot\mu^2 + (U+\alpha(x,y)\lambda\mu) ~~ (\forall~ x, y\in V',~ \lambda, \mu\in \FF)\]
for a (uniquely determined) alternating form $\alpha:V'\times V'\rightarrow \FF$, called the {\em bilinearization} of $\phi$. The polar space $\cP(\phi)$ associated to $\phi$ is defined in the same way as for quadratic forms: a point $\langle v\rangle_{V'}$ of $\PG(V')$ is a point of $\cP(\phi)$ precisely when $\phi(v) = U$ (the null element of $\FF/U$); a subspace $X$ of $\PG(V)$ is a singular subspace of $\cP(\phi)$ if and only if all of its points belong to $\cP(\phi)$. The polar space $\cP(\phi)$ is non-degenerate if and only if $\phi$ is non-degenerate, namely $\phi(x)\neq U$ for every non-zero vector $x$ in the radical $\Rad(\alpha)$ of $\alpha$.

A {\em source} of $\phi$ is a non-degenerate quadratic form $f':V'\rightarrow\FF$ with the same bilinearization as $\phi$ and such that $\phi(x) = f'(x)+U$ for every $x\in V'$. In other words, $\phi$ is $f'$ computed modulo $U$. In symbols: $\phi = |f'|_U$. As proved in \cite{Pas17}, every generalized quadratic form $\phi:V'\rightarrow\FF/U$ admits a (generally not unique) source.

We warn that the case $U = \{0\}$ is allowed in the previous setting. When $U = \{0\}$ we get back quadratic forms. We say that a generalized quadratic form $\phi:V'\rightarrow \FF/U$ is {\em proper} if $U \neq \{0\}$.

Turning back to the form $f:V\rightarrow \FF$ associated to the universal embedding $\varepsilon$ of $\cP$, with $f$ quadratic and $\ch(\FF) = 2$, let $\alpha$ be the bilinearization of $f$ and $R := \Rad(\alpha)$. All quotients of $\varepsilon$ can be obtained (modulo isomorphisms) as compositions of $\varepsilon$ with the (homomorphism of projective spaces induced by) the projection of $V$ onto a subspace of $R$. Let $\varepsilon_S$ be the embedding obtained in this way from a subspace $S\leq R$ and put $V' := V/S$ and $U := f(S) = \{f(x)\}_{x\in S}$. Then $U\FF^2\subseteq U$ (see \cite{Pas17}). Suppose that $U \neq \FF$. Then $\varepsilon_S(\cP)$ is the polar space $\cP(\phi)$ associated with the generalized quadratic form $\phi:V'\rightarrow \FF/U$, where $\phi = |f'|_U$ and $f'$ is the form induced by $f$ on a complement $W$ of $S$ in $V$, identified with $V' = V/S$ in the natural way \cite{Pas17}. On the other hand, let $U = \FF$ (which can happen only if $S = R$ and $R$ is a complement in $V$ of the span of a frame). Then $\varepsilon_S(\cP)$ is associated to the non-degenerate alternating form naturally induced by $\alpha$ on $V/R$ (see \cite{Pas17}). In this case $\dim(R) = |\FF:\FF^2|$ (see \cite{BP09}; note that in general $|\FF:\FF^2|$ is infinite).

\subsection{Back to the defect}\label{anisotropic defect}

Let $\cP$ be an embeddable non-degenerate polar space of finite rank $n \geq 2$ defined over a field $\FF$, but not a grid. As in the previous subsection, let $\varepsilon:{\cP}\rightarrow \PG(V)$ be the universal embedding of $\cP$ and $f$ an alternating, Hermitian or quadratic form $f$ of $V$ associated to $\varepsilon(\cP)$. So, $V \cong V(N,\FF)$ with $N = \er(\cP)$ and $\varepsilon(\cP) = \cP(f)$. The space $V$ admits a direct sum decomposition as follows with respect to $f$:
\begin{equation}\label{decomposition}
V ~ = ~  (V_1 \oplus V_2 \oplus ... \oplus V_n)\oplus V_0
\end{equation}
where $V_1,\dots V_n$ are mutually orthogonal hyperbolic $2$-spaces and $V_0$ is an anisotropic subspace orthogonal to each of $V_1,\dots, V_n$. If $\{u_i, v_i\}$ is a hyperbolic basis of $V_i$ for $i = 1, 2,..., n$ then $\{\{\langle u_1\rangle_V,..., \langle u_n\rangle_V\}, \{\langle v_1\rangle_V,..., \langle v_n\rangle_V\}\}$ is a frame of $\cP(f)$.

The dimension $d:=\dim(V_0)$ of $V_0$ will be called the \emph{anisotropic defect} of $f$ (also {\em defect} of $f$ for short) and denoted by $\mathrm{def}(f)$. In this subsection we shall prove that $\mathrm{def}(\cP) = \mathrm{def}(f)$.

In the sequel we will often deal with spans in $\PG(V)$. According to the conventions adopted since the beginning of this paper, we should denote them by the symbol $\langle . \rangle_{\PG(V)}$, but we prefer to replace it with a simpler symbol as $\langle . \rangle$, keeping the notation $\langle . \rangle_\cP$ and $\langle . \rangle_V$ for spans in $\cP$ and $V$ respectively. We also adopt the following notation: given a subspace $W$ of $V$ we denote by $[W]$ the subspace of $\PG(V)$ corresponding to it.

\begin{lemma}\label{1}
Let $F$ be (the set of points of) a frame of  $\cP$. Then $\langle F\rangle_{\cP} = \varepsilon^{-1}(\langle \varepsilon(F)\rangle)$.
\end{lemma}
\begin{proof}
Indeed $F$ is a frame of the polar space $\cP_F := \varepsilon^{-1}(\langle \varepsilon(F)\rangle)$ and $\varepsilon$ embeds $\cP_F$ in the $(2n-1)$-dimensional subspace $\langle \varepsilon(F)\rangle$ of $\PG(V)$. As proved in \cite{BB98}, \cite{CS97} and \cite{BC2012}, a polar space of rank $n$ associated to a non-degenerate alternating, Hermitian or quadratic form in dimension $2n$ is generated by any of its frames.
\end{proof}

\begin{corollary}\label{1bis}
Let $F$ be a frame of $\cP$. Then $\langle F\rangle_\cP = \langle F'\rangle_\cP$ for any frame $F' \subset \langle F\rangle_\cP$.
\end{corollary}
\begin{proof}
This claim is implicit in the last sentence of the proof of Lemma \ref{1}. Anyway, it trivially holds true when $\mathrm{Aut}(\cP)$ acts transitively on the set of frames of $\cP$, as it is indeed the case when $\cP$ is embeddable.
\end{proof}

By Lemma \ref{1}, if $f$ is alternating then $\cP$ is the unique nice subspace of $\cP$, whence $\mathrm{def}(\cP) = \mathrm{def}(f) = 0$. So, henceforth we assume that $f$ is either Hermitian or quadratic.

\begin{lemma}\label{2}
Let  $X$ be a nice subspace of $\cP$. Then $\langle \varepsilon(X)\rangle\cap\varepsilon({\cP})  = \varepsilon(X)$.
\end{lemma}
\begin{proof}
Let $W$ be the subspace of $V$ corresponding to $\langle \varepsilon(X)\rangle$. So, $[W] := \langle \varepsilon(X)\rangle$. The preimage $Y:= \varepsilon^{-1}([W])$ is a subspace of $\cP$ and contains $X$. We must prove that $\varepsilon(X)=\varepsilon(Y)$, namely $X = Y$.

Since $X$ is nice and $X\subseteq Y$, the subspace $Y$ is nice as well. Consequently, all frames of $X$ are frames of $Y$ as well as frames of $\cP$. As stated above, $f$ is either Hermitian or orthogonal. The polar space $\varepsilon(Y) = [W]\cap\varepsilon({\cP})$ is associated to the form induced by $f$ on  $W$, which is of the same type as $f$. We denote it by $f_Y$.

Let $\varepsilon_{|X}\colon X\rightarrow [W]$ be the embedding induced by $\varepsilon$ on $X$. According to Subsection \ref{underlying field}, the polar space $\varepsilon(X)$ is realized by either a Hermitian form or a quadratic form over $W$ or by a generalized quadratic form of $W$. Let $f_X$ be the form on $W$ associated with $\varepsilon(X)$. Since $\varepsilon (X)\subseteq \varepsilon (Y)$, we have
\begin{equation}\label{eq1 lemma2}
f_X(x) = 0 ~ \Rightarrow ~ f_Y(x) = 0, \hspace{5 mm} (\forall x\in W)
\end{equation}
where, with a little abuse of notation, we write $f_X(x)$ for $f_X(x,x)$ or $f_Y(x)$ for $f_Y(x,x)$ when $f_X$ or $f_Y$ are Hermitian. It goes without saying that, if $f_X$ is a generalized quadratic form with codomain $\FF/U$, when writing $f_X(x) = 0$ we mean that $f_X(x) = U$.

The following six cases must be considered:

\begin{enumerate}
\item\label{sc-1} Both $f_X=h$  and $f_Y=h'$ are Hermitian forms.
\item\label{sc-2} $f_Y = h$ is a Hermitian form and $f_X = q'$ is a quadratic form.
\item\label{sc-3} $f_Y = h$ is a Hermitian form and  $f_X =  |q'|_U$ is a proper generalized quadratic form, with $q'$ as a source and $\FF/U$ as the codomain.
\item\label{sc-4} $f_Y = q$ is a quadratic form and $f_X = h'$ is a Hermitian form.
\item\label{sc-5} Both $f_Y=q$  and $f_X=q'$ are quadratic forms.
\item\label{sc-6} $f_Y = q$ is a quadratic form and $f_X = |q'|_U$ is a proper generalized quadratic form.
\end{enumerate}

\noindent
In Cases \ref{sc-1} and \ref{sc-5} condition (\ref{eq1 lemma2}) implies that $h$ and $h'$ are proportional, hence $\varepsilon(X) = \varepsilon(Y)$. In Case~\ref{sc-6} condition (\ref{eq1 lemma2}) implies that
\[q'(x) = 0 ~ \Rightarrow ~ q(x) = 0,\]
which forces $q$ and $q'$ to be proportional. Hence both $q$ and $q'$ define $\varepsilon(Y)$. However, as $|q'|_U$ is a proper generalized quadratic form, the polar space defined by $|q'|_U$ contains more points than that defined by $q'$. Consequently $\varepsilon(X) \supset \varepsilon(Y)$; contradiction. So, Case \ref{sc-6} cannot occur.

Let us consider Case~\ref{sc-2}. We recall that every frame of $X$ is also a frame of $Y$. Let $F = \{\{a_1,..., a_n\},\{b_1,..., b_n\}\}$ be a frame of $X$ and let $e_1,..., e_n, f_1,..., f_n$ be representative vectors of the points $\varepsilon(a_1),..., \varepsilon(a_n), \varepsilon(b_1),..., \varepsilon(b_n)$ respectively. Clearly, $\{e_1, f_1\},..., \{e_n,f_n\}$ are mutually orthogonal hyperbolic pairs and $W$ admits an ordered basis $B = (e_1, f_1, e_2, f_2,..., e_n, f_n,...)$, where $e_1, f_1,..., e_n, f_n$ are the first $2n$ vectors. Up to rescaling, we can assume to have chosen these vectors in such a way that $h$ and $q'$ admit the following expression with respect to $B$, where $\kappa\in \FF\setminus\{0\}$ is appropriately chosen and $\sigma$ is the involutory automorphism of $\FF$ associated to $h$:
\begin{equation}
q'(x_1, x_2, x_3, x_4,\dots) = x_1x_2 + \kappa x_3 x_4 +\dots
\end{equation}
\begin{equation}\label{h}
h(x_1, x_2, x_3, x_4,\dots) = x_1^\sigma x_2 + x_2^\sigma x_1 + x_3^\sigma x_4 + x_4^\sigma x_3 + \dots
\end{equation}
Take the vector $v=(r,1,s,t,0,0,0,\dots)\in W$, where the coordinates are given with respect to $B.$
Suppose that $r+\kappa st = 0$. Clearly  $q'(v) = 0$  and by equation~\eqref{h} we get
\[h(v) = r^\sigma + r +  s^\sigma t + t^\sigma s.\]
By \eqref{eq1 lemma2}  we have $h(v) = 0$, i.e.  $r^\sigma + r +  s^\sigma t + t^\sigma s = 0$. By replacing $-\kappa st$ in place of $r$, we get
\begin{equation}\label{condition}
(\kappa st)^\sigma + \kappa st = s^\sigma t + t^\sigma s,  \hspace{5 mm} \forall s,t\in \FF.
\end{equation}
Taking $s = 1,$ condition~(\ref{condition}) becomes $\mathrm{Tr}((\kappa-1)t)=0$ for every $t\in \FF$ forcing  $\kappa=1.$
 Thus, with $\kappa=1$, condition~(\ref{condition})  becomes
\[(st)^\sigma + st = s^\sigma t + t^\sigma s,  \hspace{5 mm} \forall s,t\in \FF.\]
Equivalently,
\[(s^\sigma-s)(t^\sigma-t) = 0, \hspace{5 mm} \forall s,t \in \FF.\]
This is not possible. Hence Case~\ref{sc-2} cannot occur.

Suppose Case \ref{sc-3} holds. For any $x\in W$ such that $q'(x)=0$ we have $|q'(x)|_U = 0$ and,  by \eqref{eq1 lemma2}, also $h(x) = 0$. We can now apply the same argument as in Case \ref{sc-2} on $q'$ and we reach a contradiction, as before.

Suppose  Case~\ref{sc-4} holds. Analogously to Case~\ref{sc-2} suppose that the expressions of $q$ and $h'$ with respect to a given basis $B$ of $W$ are the following, where $\kappa\in \FF\setminus\{0\}$ is appropriately chosen
\begin{equation}\label{q}
q(x_1, x_2, x_3, x_4,\dots) = x_1x_2 + \kappa x_3 x_4 +\dots
\end{equation}
\begin{equation}
h'(x_1, x_2, x_3, x_4,\dots) = x_1^\sigma x_2 + x_2^\sigma x_1 + x_3^\sigma x_4 + x_4^\sigma x_3 + \dots
\end{equation}
Take the vector $v=(r,1,s,t,0,0,0,\dots)\in W$. Suppose that $r+s^\sigma t = 0$. Clearly  $h'(v) = 0$  and by equation~(\ref{q}) we get
\[q(v) = r+\kappa st.\]
By \eqref{eq1 lemma2}  we have   $q(v) = 0$, i.e.    $r=-\kappa s t.$ Thus, we obtain the condition
\[s^\sigma t = \kappa st, \hspace{5 mm} \forall s,t\in \FF.\]
Taking  $s = t = 1$  we get $\kappa = 1$, so the previous condition becomes
\[s^\sigma t = st, \hspace{5 mm} \forall s,t\in \FF\]
forcing $\sigma$ to be the identity. This is a contradiction. Hence Case~\ref{sc-4} is impossible.

The lemma is proved.
\end{proof}

\begin{corollary}\label{2 bis}
The following are equivalent for two nice subspaces $X, Y$ of $\cP$ with $X\subset Y$:
\begin{enumerate}[{\rm (1)}]
\item\label{Ev1} $\langle X\cup\{x\}\rangle_\cP = Y$ for some point $x\in Y\setminus X$.
\item\label{Ev2} $X$ is a maximal subspace of $Y$;
\item\label{Ev3} $X$ is a hyperplane of $Y$:
\item\label{Ev4} $\langle \varepsilon(X)\rangle$ is a hyperplane of $\langle \varepsilon(Y)\rangle$.
\end{enumerate}
\end{corollary}
\begin{proof}
It is well known that (\ref{Ev4}) implies (\ref{Ev3}) which in turn implies (\ref{Ev2}). Trivially, (\ref{Ev2}) implies (\ref{Ev1}). It remains to prove that (\ref{Ev1}) implies (\ref{Ev4}). Suppose that $\langle X\cup\{x\} \rangle_\cP = Y$ for a point $x\in Y\setminus X$. Then $\langle \varepsilon(X)\cup\{\varepsilon(x)\}\rangle = \langle\varepsilon(Y)\rangle$. However $\varepsilon(x)\not \in \langle \varepsilon(X)\rangle$ by Lemma \ref{2}, since $x\not \in X$. Therefore  $\langle \varepsilon(X)\rangle$ is a hyperplane of $\langle \varepsilon(Y)\rangle$, as claimed in (\ref{Ev4}).
\end{proof}

\begin{note}\label{2 ter}
The conclusion of Lemma \ref{2} also holds for $X$ a possibly degenerate subspace of $\cP$ with $\rank(\Res(X^\perp)^\uparrow)) \geq 2$. Accordingly, Corollary \ref{2 bis} can be stated in a more general form. Most likely, both Lemma \ref{1} and Lemma \ref{2} as well as their corollaries \ref{1bis} and \ref{2 bis} hold in a more general setting, with the field $\FF$ replaced by any division ring, but the technical details of the proofs look quite laborious; we have not checked them. We leave this job for future work.
\end{note}

In the sequel, given a chain $C = (X_i)_{i\in I}$ of $\mathfrak{N}(\cP)$, we denote by $W_i$ the subspace of $V$ corresponding to the span $\langle \varepsilon(X_i)\rangle$ of $\varepsilon(X_i)$ in $\PG(V)$. Clearly, if $C$ is maximal then it admits a minimum as well as a maximum element; its minimum element is spanned by a frame of $\cP$ while its maximum is $\cP$ itself.
We index the elements of a well ordered chain $C$ by ordinal numbers. In particular, when $C$ is known to admit a maximum element we write $C = (X_\delta)_{\delta\leq \omega}$ for an ordinal number $\omega$ such that $C$ and $\{\delta\}_{\delta\leq \omega}$ are isomorphic as ordered sets, $X_\omega$ being the maximum of $C$.

\begin{lemma}\label{3 new}
A well ordered chain $C = (X_\delta)_{\delta\leq \omega}$ of $\mathfrak{N}(\cP)$ with $X_\omega = \cP$ is maximal as a well ordered chain if and only if it is maximal as a chain, if and only if all the following hold:
\begin{enumerate}[{\rm (1)}]
\item\label{Ec1} $X_0$ is spanned by any of its frames;
\item\label{Ec2} for $\delta < \omega$, $X_\delta$  is a maximal subspace of $X_{\delta+1}$;
\item\label{Ec3} for every limit ordinal $\gamma\leq \omega$, we have $\bigcup_{\delta < \gamma}X_\delta = X_\gamma$.
\end{enumerate}
\end{lemma}
\begin{proof}
Suppose that the above conditions are satisfied. By way of contradiction, let $C'$ be a chain properly containing $C$. Then $X'\not\in C$ for some term $X'$ of $C'$. However, as $C \subset C'$, for every $\delta \leq \omega$ either $X' \subset X_\delta$ or $X_\delta \subset X'$. Moreover, $X'\subset X_\omega = \cP$ and $X_0\subset X'$, by Condition~(\ref{Ec1}). Let $\gamma_1$ be the least ordinal $\gamma$ such that $X_\gamma \supset X'$. Then $X_\delta \subset X'$ for every $\delta < \gamma_1$. If $\gamma_1 = \gamma_0+1$ then $X'$ sits between $X_{\gamma_0}$ and $X_{\gamma_1}$. This contradicts (\ref{Ec2}). Therefore $\gamma_1$ is a limit ordinal. However in this case $X'$ contains $\bigcup_{\delta < \gamma_1}X_\delta$, hence $X' \supseteq X_{\gamma_1}$ by (\ref{Ec3}). Again, we have reached a contradiction. So, we must conclude that $C$ is maximal in the set of all chains of $\mathfrak{N}(\cP)$.

Conversely, suppose that $C$ is maximal as a well ordered chain. If $\langle F\rangle_\cP \subset X_0$ for some frame $F$ of $X_0$, then we can add $\langle F\rangle_\cP$ in front of $C$ as the new initial element, thus contradicting the maximality of $C$. Therefore (\ref{Ec1}) must hold. Similarly, if $X_\delta$ is not maximal as a proper subspace of $X_{\delta+1}$, let $X_\delta \subset X \subset X_{\delta+1}$. Necessarily $\rank(X) = n$, as $\rank(X_\delta) = n$. If $X$ is degenerate then $X^\perp\cap X_\delta = \emptyset$, since $X_\delta$ is non-degenerate. It follows that $X$ contains singular subspaces of rank $n+1$; contradiction. Therefore $X$ is nice. Thus we can add $X$ to $C$, inserting it between $X_\delta$ and $X_{\delta+1}$. Once again, we contradict the maximality of $C$. So, (\ref{Ec2}) must hold. Finally, suppose that $X := \bigcup_{\delta < \gamma}X_\delta \subset X_\gamma$. Clearly, $X$ is nice. So we can insert it in $C$ just before $X_\gamma$, after all terms $X_\delta$ with $\delta < \gamma$. Again, we have reached a contradiction; condition (\ref{Ec3}) must hold too.
\end{proof}

Note that in the proof of Lemma \ref{3 new} we have made no use of the hypothesis that $\cP$ is embeddable and defined over a field. We exploit that hypothesis in the next corollary, where conditions (\ref{Ec1}), (\ref{Ec2}) and (\ref{Ec3}) of Lemma \ref{3 new} are rephrased for the subspaces $W_\delta$ of $V$ corresponding to the terms $X_\delta$ of $C$.

\begin{corollary}\label{3 new bis}
  A well ordered chain $C = (X_\delta)_{\delta\leq \omega}$ of $\mathfrak{N}(\cP)$ with $X_\omega = \cP$ is maximal if and only if all the following conditions are satisfied:
\begin{enumerate}[$(1')$]
\item\label{Ec1'} $\dim(W_0) = 2n$;
\item\label{Ec2'} for $\delta < \omega$, $W_\delta$ is a hyperplane of $W_{\delta+1}$;
\item\label{Ec3'} for every limit ordinal $\gamma\leq \omega$, we have $\bigcup_{\delta < \gamma}W_\delta = W_\gamma$.
\end{enumerate}
\end{corollary}
\begin{proof}
Conditions $(\ref{Ec1'}')$ and $(\ref{Ec2'}')$ are equivalent to (\ref{Ec1}) and (\ref{Ec2}) of Lemma \ref{3 new} in view of Corollaries \ref{1bis} and \ref{2 bis} respectively. Condition $(\ref{Ec3'}')$ is equivalent to (\ref{Ec3}) by Lemma \ref{2}.
\end{proof}

We are now going to prove the equality $\mathrm{def}(\cP) = \mathrm{def}(f)$.
We firstly consider the finite dimensional case. In this case, as stated in the next theorem, all chains of $\mathfrak{N}(\cP)$ are finite, whence well ordered. So, we can forget about well ordering and simply refer to chains.

\begin{theorem}\label{3}
Let $N = 2n+d$ be finite, namely $d = \mathrm{def}(f)$ is finite. Then all chains of $\mathfrak{N}(\cP)$ have length at most $d$, the maximal chains being those of length $d$.
\end{theorem}
\begin{proof}
We only must prove that all chains of $\mathfrak{N}(\cP)$ are finite of length at most $d$. Having proved this, the conclusion immediately follows from Corollary \ref{3 new bis}. So, let $(X_i)_{i\in I}$ be a chain of $\mathfrak{N}(\cP)$ with $I$ a totally ordered set of indices. With $i, j \in I$ and $W_i, W_j$ defined as above, if $i < j$ then $W_i \subset W_j$ by Lemma \ref{2}. Therefore $|I| \leq N-2n = d$.
\end{proof}

\begin{theorem}\label{4}
Let $d$ be infinite. Then:
\begin{enumerate}[{\rm (1)}]
\item\label{Mx1} All maximal well ordered chains of $\mathfrak{N}(\cP)$ have length $d$.
\item\label{Mx2} Every well ordered chain of $\mathfrak{N}(\cP)$ is contained in a maximal well ordered chain.
\end{enumerate}
\end{theorem}
\begin{proof}
By Corollary \ref{3 new bis}, if $C = (X_\delta)_{\delta\leq \omega}$ is a maximal well ordered chain of $\mathfrak{N}(\cP)$ then the codimension of $W_0$ in $V$ is equal to the cardinality $|\omega|$ of $\omega$, which is just the length of $C$. On the other hand, $\dim(W_0) = 2n$, since $X_0$ is spanned by a frame. It follows that $\dim(V) = 2n+|\omega|$, namely $2n+|\omega| = N = 2n+d$. Therefore $|\omega| = d$ ($= 2n+d$ since $d$ is infinite). Claim (\ref{Mx1}) is proved.

Turning to (\ref{Mx2}), let $C$ be a well ordered chain of $\mathfrak{N}(\cP)$. With no loss, we can assume that $\cP\in C$. So, $C  = (X_\delta)_{\delta\leq \omega}$ where $X_\omega = \cP$. We can also assume that $X_0$ is spanned by a frame. Indeed, if not, then we can replace $C$ with the chain obtained by adding the span of a frame of $X_0$ as the new initial element.  By Lemma \ref{2}, we have $W_\delta \subset W_{\delta+1}$ for every $\delta < \omega$. Let $B_\delta = (b_{\delta,\xi})_{\xi < \omega_\delta}$ be a well ordered basis for a complement of $W_\delta$ in $W_{\delta+1}$ and, for every $\chi \leq \omega_\delta$, put $W_{\delta, \chi} := \langle W_\delta\cup \{b_{\delta,\xi}\}_{\xi < \chi}\rangle_V$ and $X_{\delta,\chi} = \varepsilon^{-1}([W_{\delta,\chi}])$. So, $W_{\delta, 0} = W_\delta$ and $W_{\delta,\omega_\delta} = W_{\delta+1}$. Accordingly, $X_{\delta,0} = X_\delta$ and $X_{\delta,\omega_\delta} = X_{\delta+1}$ by Lemma \ref{2}. Moreover $X_{\delta,\chi} \subset X_{\delta,\chi'}$ for $0 \leq \chi < \chi' \leq \omega_\delta$, again by Lemma \ref{2}.

Let $\gamma\leq \omega$ be a limit ordinal and put $W'_\gamma := \cup_{\delta < \gamma}W_\delta$. Clearly $W'_\gamma \subseteq W_\gamma$. Suppose that $W'_\gamma \subset W_\gamma$ and let $B_\gamma = (b_{\delta,\xi})_{\xi < \omega_\gamma}$ be a well ordered basis for a complement of $W'_\gamma$ in $W_\gamma$. As above, put $W_{\gamma, \chi} := \langle W'_\gamma \cup \{b_{\delta,\xi}\}_{\xi < \chi}\rangle_V$ and $X_{\delta,\chi} = \varepsilon^{-1}([W_{\delta,\chi}])$ for every $\chi \leq \omega_\gamma$. So $W_{\gamma, 0} = W'_\gamma$, $X_{\gamma, 0} = \cup_{\delta < \gamma}X_\delta$, $W_{\gamma, \omega_\gamma} = W_\gamma$, $X_{\gamma, \omega_\gamma} = X_\gamma$ and $X_{\gamma,\chi} \subset X_{\gamma,\chi'}$ for any $0 \leq \chi < \chi' \leq \omega_\gamma$.

We can enrich the chain $C$ by adding the spaces $X_{\delta, \chi}$ and $X_{\gamma, \chi}$ defined as above, namely we place the chain $(X_{\delta,\chi})_{0 < \chi < \omega_\delta}$ between $X_\delta$ and $X_{\delta+1}$ and, when $W'_\gamma \subset W_\gamma$, we place $(X_{\delta,\chi})_{\chi < \omega_\gamma}$ after all terms $X_\delta$ with $\delta < \gamma$ and just before $X_\gamma$. In this way we obtain a larger well ordered chain $C'$. (Recall that a well ordered union of pairwise disjoint well ordered sets is still well ordered.) The chain $C'$ satisfies the conditions of Corollary \ref{3 new bis}. Hence it is a maximal well ordered chain.
\end{proof}

By Theorems \ref{3} and \ref{4} we immediately obtain the main result of this subsection:

\begin{corollary}\label{3&4}
In any case, $\mathrm{def}(\cP) = \mathrm{def}(f)$.
\end{corollary}

\begin{note}\label{monster chains}
When $d$ is infinite $\mathfrak{N}(\cP)$ admits chains of length greater than $d$. By Theorem \ref{4}, these chains are not well ordered. For instance, let $d = \aleph_0$ and, with $V_1,..., V_n, V_0$ as in decomposition \eqref{decomposition}, let $B = (b_r)_{r\in\QQ}$ be a basis of $V_0$, indexed by the set $\QQ$ or rational numbers. Let $W_{-\infty} = \oplus_{i=1}^nV_i$ and for every real number $a$ put $W_a := \langle W_{-\infty}\cup\{b_r\}_{r < a}\rangle_V$ and $X_a := \varepsilon^{-1}([W_a])$. Also $X_{-\infty} := \varepsilon^{-1}([W_{-\infty}])$ and $X_{+\infty} := \cP$. Then $(X_a)_{a\in\RR}$, with $X_\infty$ and $X_{+\infty}$ added as the least and the greatest element, is a chain of $\mathfrak{N}(\cP)$ of length $|\RR|  > d$.

This chain is not maximal but, by Zorn's Lemma, it is contained in a maximal one. In fact it is contained in a unique a maximal chain, constructed as follows. For every $r\in \QQ$ put $W'_r := \langle W_r, b_r\rangle_V$, $X'_r := \varepsilon^{-1}([W'_r])$ and place $X'_r$ just after $X_r$, before all $X_a \in C$ with $a > r$. The chain $C'$ constructed in this way is maximal and contains $C$. Clearly, $|C'| = |C| = |\RR|$.
\end{note}

\begin{note}
If $d$ is countable then all maximal chains of $\mathfrak{N}(\cP)$ have length at least $d$. We conjecture that the same is true when $d > \aleph_0$. If so, then we can replace Definition \ref{defect} with a nicer one, where no mention of well ordered chains is made, e. g. as follows: the defect of $\cP$ is the minimum length of a maximal chain of $\mathfrak{N}(\cP)$. Regretfully, we presently do not know how to prove that conjecture.
\end{note}

\subsection{The polar Grassmannians considered in this paper}

Throughout this paper $V = V(N,\FF)$ is a vector space of (possibly infinite) dimension $N$ equipped with a non-degenerate Hermitian or quadratic form $f$ of finite Witt index $n \geq 2$ and defect $d = N-2n$, and ${\cal P} = \cP(f)$ is the non-degenerate polar space associated to $f$. So, $\rank(\cP) = n$ and $\mathrm{def}(\cP) = d$. In many of the results we will obtain we allow $d$ to be infinite, although some of them loose much of significance when $d$ is infinite.

For $1\leq k < n$ the polar $k$-Grassmannian ${\cal P}_k$ is a full subgeometry (but not a subspace) of the $k$-Grassmannian $\cG_k$ of $\PG(V)$. On the other hand, the point-set $S_n(\cP)$ of $\cP_n$ is a subset of $\cG_n$ but $\cP_n$ is never a full subgeometry of $\cG_n$. Indeed the lines of ${\cal P}_n$ are not contained in lines of $\cG_n$ except when $d = 0$. When $d = 0$ and $f$ is Hermitian the lines of $\cP_n$ are Baer sublines of lines of $\cG_n$ while when $d = 0$ and $f$ is quadratic all lines of $\cP_n$ are thin (each of them has just two points). In this case no three points of $\cP_n$ are collinear in $\cG_n$.

\paragraph{Terminology and notation}
We call $\cP(f)$ and its $k$-Grassmannian a {\em Hermitian polar space} and a {\em Hermitian $k$-Grassmannian} or a {\em quadratic polar space} and a {\em quadratic $k$-Grassmannian} according to whether $f$ is Hermitian or quadratic. We will use the letters $\cH$ and $\cH_k$ to denote Hermitian polar spaces and Hermitian $k$-Grassmannians respectively and letters as $\cQ$ and $\cQ_k$ for quadratic polar spaces and quadratic $k$-Grassmannians.
Recall that $\cH_1 = \cH$ and $\cQ_1 = \cQ$.

\section{Preliminaries on generation of polar Grassmannians}\label{sec3new}

Throughout this section $\cP$ is a non-degenerate polar space of finite rank $n \geq 2$ with no thin lines and $d := \mathrm{def}(\cP)$. When needed, we will assume the following, which we call the {\em Maximal Well Ordered Chain} condition (also {\em Maximal Chain} condition or (MC), for short):

\begin{itemize}
\item[(MC)] Every nice subspace of $\cP$ belongs to a maximal well ordered chain of $\mathfrak{N}(\cP)$.
\end{itemize}
Obviously, (MC) holds when $\mathrm{def}(\cP)$ is finite. In view of Theorem \ref{4}, when $\cP$ is embeddable and defined over a field then  (MC) is satisfied whatever $\mathrm{def}(\cP)$ is. However (MC) certainly holds true under far weaker hypotheses. Most likely, it holds whenever $n > 2$ or $n = 2$ and $\cP$ is embeddable.

According to the notation introduced in Section \ref{basics}, if $K$ is a subspace of $\cP$ of rank $m \leq n$ and $1 \leq k \leq m$, we put $S_k(K) := \{X\in S_k(\cP) : X\subset K\}$. The set $S_k(K)$ is a subspace of $\cP_k$. We denote by $\cP_k(K)$ the geometry induced by $\cP_k$ on $S_k(K)$. When $K$ is non-degenerate we can also consider the $k$-Grassmannian $K_k$. If either $m > k$ or $k = m = n$ then $\cP_k(K) = K_k$. On the other hand, if $m = k < n$ then $\cP_k(K) \neq K_k$, although these two geometries have the same set of points. Indeed in this case $S_k(K)$ is a co-clique in the collinearity graph of $\cP_k$, whence $\cP_k(K)$ has no lines, while $K_k$ does admit lines.

We also adopt the following convention: if $X$ is a singular subspace of $\cP$ of rank $h < k$ then we denote by $S_k(X)$ the set of $k$-subspaces of $\cP$ containing $X$. In other words, $S_k(X) = S_k(\cP)\cap \Res(X)^\uparrow$.

\subsection{Two extremal cases: $k = n$ and $k = 1$}

\subsubsection{The case $k = n$}\label{dps-introduction}

\begin{prop}\label{k=n; 0}
Let $H$ be a hyperplane of $\cP$ and suppose that the polar space $H$ is non-degenerate of rank $n$ (in short, $H$ is a nice subspace of $\cP$). Then $\langle S_n(H)\rangle_{\cP_n} = \cP_n$.
\end{prop}
\begin{proof}
  Let $X\in S_n(\cP)$. Then either $X\in S_n(H)$ or $X\cap H \in S_{n-1}(H)$. However $X\cap H$ is contained in at least two $n$-spaces $X_1, X_2\in S_n(H)$. It follows that $X$ belongs to the line $\langle X_1, X_2\rangle_{\cP_n} = \Res(X_1\cap X_2)^\uparrow$ of $\cP_n$.
\end{proof}
\begin{corollary}\label{clkn}
  Under the hypotheses of Proposition \ref{k=n; 0}, we have $\gr(\cP_n)\leq \gr(H_n).$
\end{corollary}
\begin{corollary}\label{clkn bis}
Suppose that $\cP$ satisfies the Maximal Chain condition {\rm (MC)} and let $K$ be a nice subspace of $\cP$. Then
  $\langle S_n(K)\rangle_{\cP_n}=\cP_n$ and so $\gr(\cP_n)\leq \gr(K_n)$.
\end{corollary}
\begin{proof}
By (MC) there exists a well ordered chain $(K_\delta)_{\delta\leq \omega}$ with $K_0 = K$, $K_\omega = \cP$ and such that both conditions (\ref{Ec2}) and (\ref{Ec3}) of Lemma \ref{3 new} hold (recall that Lemma \ref{3 new} holds even if $\cP$ is neither embeddable nor defined over a field). With the help of Corollary \ref{clkn}, by a routine inductive argument we obtain that $\langle S_n(K)\rangle_{\cP_n}= S_n(K_\delta)$ for every $\delta \leq \omega$. In particular, $\langle S_n(K)\rangle_{\cP_n}= S_n(K_\omega) = S_n(\cP)$.
\end{proof}

\begin{note}
With $\cP$ and $H$ as in Proposition \ref{k=n; 0}, it can be that $\gr(\cP_n) < \gr(H_n)$. For instance, let $\cP$ be Hermitian of defect $d = 1$. Then $H$ has defect $0$. Let $F$ be a frame of $H$ and $F_n$ the family of $n$-subspaces of $H$ spanned by points of $F$. Then $F_n\subset S_n(\cP)$ spans $\cP_n$ (see \cite{CS01}). So, $\gr(\cP_n) \leq 2^n$, as $|F_n| = 2^n$. Nevertheless $F_n$ is not big enough to generate $H_n$. Indeed, as proved in \cite{BC2012}, if the underlying field $\FF$ of $\cP$ has order $|\FF| > 4$ then $\gr(H_n)$ is equal to ${{2n}\choose n}$ ($ > 2^n$); when $\FF = \FF_4$ and $n > 2$ it is even larger, since in this case $\er(H_n) > {{2n}\choose n}$ (Li \cite{Li02}).

Similarly, let $\cP$ be orthogonal of defect $1$. Then $H$ has defect $0$. Let $F$ be a frame of $H$. If the underlying field $\FF$ of $\cP$ has characteristic $\ch(\FF) \neq 2$ then $F_n$ spans $\cP_n$ (see \cite{BB98}), but it cannot span $H_n$. Indeed $S_n(H)$ is the unique generating set for $H_n$.
\end{note}

\subsubsection{The case $k = 1$}

\begin{prop}\label{t:k=1}
Suppose that $\mathfrak{N}(\cP)$ admits at least one maximal well ordered chain. Then $\gr(\cP) \leq 2n+d$ ($= d$ if $d$ is infinite).
\end{prop}
\begin{proof}
Let $C = (K_\delta)_{\delta\leq \omega}$ be a maximal well ordered chain of $\mathfrak{N}(\cP)$. We shall prove by (transfinite) induction that
\begin{itemize}
\item[$(\ast)$] $\gr(K_\delta) \leq 2n+|\delta|$ for every $\delta \leq \omega$.
\end{itemize}
The conclusion of the proposition is just the case $\delta = \omega$ of $(\ast)$. Recall that $C$ satisfies all conditions of Lemma \ref{3 new}. Hence $\gr(K_0) \leq 2n$ by condition (1) of that lemma. Let now $\gamma \leq \omega$ and suppose that $(\ast)$ holds on $K_\delta$ for every $\delta < \gamma$.

Suppose firstly that $\gamma$ is not a limit ordinal, say $\gamma = \delta+1$.  By (\ref{Ec2}) of Lemma \ref{3 new},  $K_\delta$ is a maximal subspace of $K_\gamma$. Hence $\gr(K_\gamma) \leq \gr(K_\delta)+1$. However $\gr(K_\delta) \leq 2n+|\delta|$ by the inductive hypothesis. Therefore $\gr(K_\gamma) \leq 2n+|\delta|+1 = 2n+ |\gamma|$.

On the other hand, let $\gamma$ be a limit ordinal. Then $K_\gamma = \bigcup_{\delta < \gamma}K_\delta$ by (3) of Lemma \ref{3 new}. For every $\delta < \gamma$, let $G_\delta$ be a generating set for $K_\delta$. Then $G_\gamma := \bigcup_{\delta < \gamma}G_\delta$ is a generating set for $K_\gamma$. However $(\ast)$ holds for $K_\delta$ for every $\delta < \gamma$, by the inductive hypothesis. Accordingly, we can assume to have chosen
$G_\delta$ in such a way that $|G_\delta| \leq 2n+|\delta|$. Clearly, $2n+|\delta| \leq 2n+ |\gamma| = |\gamma|$. (Indeed $\gamma$, being a limit ordinal, is infinite.) So, $G_\gamma$ is the union  of $|\gamma|$ sets all of which have cardinality at most $|\gamma|$. Therefore $|G_\gamma| \leq |\gamma|$ (because $\gamma$ is infinite). However $|\gamma| = 2n+|\gamma|$. The proof is complete.
\end{proof}

By Proposition \ref{t:k=1} we immediately obtain the following:

\begin{corollary}\label{t:k=1 symplectic even}
Under the hypotheses of Proposition \ref{t:k=1}, if moreover $\cP$ admits an embedding of dimension $2n+d$, then $\gr(\cP) = \er(\cP) = 2n+d$.
\end{corollary}

\subsection{Generation of $k$-polar Grassmannians with $1 < k < n$}\label{k<n-introduction}

\subsubsection{The case $k=2$}\label{k=2}
Suppose that $n > 2$. The following lemma is well known (see e.g. \cite{BPa01}). We will use it very often in the sequel.

\begin{lemma}\label{lemma0}
Let $G \subseteq S_2(\cP)$. If a singular plane $X\in S_3(\cP)$ contains three elements of $\langle G \rangle_{\cP_2}$ which do not belong to the same pencil of lines of $X$, then $S_2(X) \subseteq\langle G \rangle_{\cP_2}$.
\end{lemma}
 Given a (possibly singular) hyperplane $H$ of $\cP$, let $\ell_0$ be a line of $\cP$ and $p_0$ a point of $\cP\setminus H$ such that $\ell_0\not\subseteq H\cup p_0^\perp$ and $H\cap \ell_0\not= p_0^\perp\cap \ell_0.$ Recall that, according to conventions stated at the beginning of this section, $S_2(p_0)$ is the set of lines of $\cP$ through $p_0$ and $S_2(H)$ is the set of lines of $\cP$ which are contained in $H$. We put:
\begin{equation}\label{defgensetk=2}
S_2(H,p_0,\ell_0) ~ := ~  S_2(H)\cup S_2(p_0)\cup \{\ell_0\} ~ \subseteq ~  S_2(\cP).
\end{equation}

\begin{proposition}\label{lemma1}
The set $S_2(H, p_0, \ell_0)$ spans $\cP_2.$
\end{proposition}
\begin{proof}
Put $G := S_2(H,p_0,\ell_0)$ for short. Let $m\in \cP_2$. We will show that $m\in \langle G \rangle_{\cP_2}$. The proof is divided into several steps according to the different mutual positions between $m$ and $p_0$ or $m$ and $\ell_0.$ If $p_0\in m$ or $m\subseteq H$ there is nothing to prove. Suppose $p_0\not\in m$ and $m\not\subseteq H.$ There are five main cases to consider.

\medskip

\noindent
(a)  {\it $m\perp p_0$, namely $m$ is coplanar with $p_0$ in $\cP$.}  The plane $X:=\langle p_0, m\rangle_{\cP}$ of $\cP$ spanned by $p_0$ and $m$ contains at least two lines through $p_0$, which are in $S_2(p_0)\subseteq G$, and the line $X\cap H\in S_2(H)\subseteq G$, which does not contain $p_0$. Hence $m\in \langle G\rangle_{\cP_2}$ by Lemma~\ref{lemma0}.

\medskip

\noindent
(b)  {\it $m\not\perp p_0$, $m\cap \ell_0\not= \emptyset$ and $m$ is coplanar with $\ell_0$ in $\cP$.} Let $X:=\langle \ell_0, m\rangle_{\cP}$ be the plane of $\cP$ spanned by $\ell_0$ and $m$. Then $p_0^\perp\cap X$ is a line of $\cP$ coplanar with $p_0$. It belongs to $\langle G\rangle_{\cP_2}$ by Case (a). Since $H\cap \ell_0\not= p_0^\perp\cap \ell_0$ by assumption, the three lines $\ell_0$, $X\cap H$ and $p_0^\perp\cap X$ are non-concurrent lines of $X$. The first two of them belong to $G$ while $p_0^\perp\cap X \in \langle G\rangle_{\cP_2}$ by Case (a). Hence all lines of $X$ belong to $\langle G\rangle_{\cP_2}$ by Lemma~\ref{lemma0}. In particular, $m\in\langle G\rangle_{\cP_2}$.

\medskip

\noindent
(c) {\it $m\not\perp p_0$, $m\cap \ell_0\not=\emptyset$ but $m$ is not coplanar with $\ell_0$.} Put $p:=m\cap \ell_0.$ There are four subcases to consider.

\medskip

\noindent
(c.1) {\it $p\in H$ and $X\cap H\neq \ell_0^\perp \cap X$ for at least one singular plane $X$ through $m$.} Then the line $\ell_0^\perp \cap X$ belongs to $\langle G\rangle_{\cP_2}$ by Case (b). Moreover $X\cap H\in S_2(H)$, the lines $\ell_0^\perp\cap X$ and $X\cap H$ are collinear as points of $\cP_2$ and $m\in\langle X\cap H, \ell_0^\perp \cap X\rangle_{\cP_2}$. Therefore $m \in\langle G \rangle_{\cP_2}$.

\medskip

\noindent
(c.2) {\it $p\in H$ and $X\cap H=\ell_0^\perp \cap X$ for every singular plane $X$ through $m$.} In $\Res(p)$, the set $H_p$ of all lines of $H$ through $p$ reads as a hyperplane and $\{m,\ell_0\}^{\perp_p}\subseteq H_p$, where $\perp_p$ stands for the collinearity relation in the polar space $\Res(p)$. Let
$m'\neq m$ be a line through $p$ coplanar with $m$. Suppose $X\cap H={m'}\,^{\perp}\cap H$ for any singular plane $X$ of $\cP$ through $m$;  then in $\Res(p)$ we have $\{m,\ell_0\}^{\perp_p}=\ell_0^{\perp_p}\cap H_p=m^{\perp_p}\cap H_p=m'\,^{\perp_p}\cap H_p$. So $m'\in\{m,\ell_0\}^{\perp_p\perp_p}$. This is a contradiction. Indeed $\{m,\ell_0\}^{\perp\perp}$ is a hyperbolic line of $\Res(p)$ because $m\not\perp_p\ell_0$, while $m' \perp_p m \neq m'$.

So, for each line $m'\neq m$ through $p$ coplanar with $\ell_0$ there exists at least a singular plane $X$ through $m$ such that $X\cap H\neq m'^\perp \cap \alpha.$ Consequently $m'\in \langle G \rangle_{\cP_2}$ by Case (c.1). Consider now the singular plane $X':=\langle m', m'^\perp \cap \alpha \rangle_{\cP}$ of $\cP$ spanned by $m'$ and $m'^\perp \cap \alpha.$ The plane $X'$ contains $m'\in \langle G\rangle_{\cP_2}$; moreover $X'\cap H\in S_2(H)$ and $p_0^\perp \cap X' \in \langle G\rangle_{\cP_2}$ by Case (a). Therefore all lines of $X'$ belong to $\langle G\rangle_{\cP_2}$ by Lemma~\ref{lemma0}. In particular, $m'\,^\perp \cap X \in \langle G\rangle_{\cP_2}.$ However, $\alpha\cap H$ and $m'\,^\perp\cap X$ are collinear as points of $\cP_2$ and $m$ belongs to the line of $\cP_2$ spanned by them. It follows that $m\in \langle G\rangle_{\cP_2}$.

\medskip

\noindent
(c.3) {\it $p\not\in H$ and $p_0^\perp\cap X\cap H\not= m^\perp \cap X\cap H$ for some singular plane $X$ of $\cP$ through $\ell_0$.} As $p\in m\cap X$,  both intersections $p_0^{\perp}\cap\pi\cap H$ and $m^{\perp}\cap\pi\cap H$ are single points. Consider the singular plane $Y :=\langle m,m^\perp\cap X\rangle_{\cP}$. Then $Y$ contains the following three non-concurrent lines: $m^\perp\cap X$, which is in $\langle G\rangle_{\cP_2} $ by Case (b), $Y\cap H\in G$ and $Y\cap p_0^\perp$, which is in $\langle G \rangle_{\cP_2} $ by Case (a). By Lemma~\ref{lemma0}, all the lines of $Y$ belong to $\langle G\rangle_{\cP_2}$. In particular, $m\in \langle G\rangle_{\cP_2}.$

\medskip

\noindent
(c.4) {\it $p\not\in H$ and $p_0^\perp\cap X\cap H= m^\perp \cap X\cap H$ for every singular plane $X$ of $\cP$ through $\ell_0$.} Let $M$ be the set of lines $n$ through $p$ not coplanar with $\ell_0$ but coplanar with $m$ (possibly $n=m$) and $S_3(\ell_0)$ the set of singular planes through $\ell_0$. Suppose firstly that $p_0^\perp\cap X\cap H =  n^\perp \cap X\cap H$ for every $n\in M$ and every $X\in S_3(\ell_0)$. The intersection
$q_X := p_0^\perp\cap X\cap H =  n^\perp \cap X\cap H$ is a point. This point does not depend on the choice of $n\in M$ but it depends on the choice of $X\in S_3(\ell_0)$, namely elements $X, Y\in S_3(\ell_0)$ exist such that $q_X \neq q_Y$. Indeed if otherwise
$q_X =\ell_0\cap H$, which contradicts the hypothesis that $\ell_0\cap p_0^{\perp}\neq\ell_0\cap H$.

We have $\ell_0\subseteq q_X^{\perp}$, since $q_X\in X$ and $X \supset\ell_0$. Furthermore $n\subset q_X^{\perp}$  for every $n\in M$. Consequently, $p \perp q_X$. In the polar space $\Res(p)$ the lines through $p$ coplanar with $\ell_0$ form a singular hyperplane $\ell_0^{\perp_p}$ while $M$ reads as $m^{\perp_p}\setminus \ell_0^{\perp_p}$. Let now $\ell_X$ be the line of $\cP$ joining $p$ with $q_X$. As $p\in n$ and $q_X\perp n$ for every $n\in M$, we have $M\subseteq \ell_X^{\perp_p}$. However, $M$ spans the hyperplane $m^{\perp_p}$ of $\Res(p)$. It follows that $m^{\perp_p} = \ell_X^{\perp_p}$, whence $m = \ell_X$. Therefore $q_X\in m$. Accordingly, $q_X=m\cap H$ as $q_X\in H$. So, $q_X$ does not depend on the choice of $X\in S_3(\ell_0)$. But this contradicts what we have proved above.

Therefore $p_0^\perp\cap X\cap H \neq  n^\perp \cap X\cap H$ for some $n\in M$ and $X\in S_3(\ell_0)$. Let $Y$ be the plane of $\cP$ spanned by $m$ and $n$. We have $n\in \langle G\rangle_{\cP_2}$ by Case (c.3) and $\ell_0^\perp \cap Y \in \langle G\rangle_{\cP_2}$ by Case (b). Moreover $n$ and  $\ell_0^\perp \cap Y$ are collinear as points of $\cP_2$ and $m$ belongs to the line of $\cP_2$ spanned by them. Hence $m\in \langle G\rangle_{\cP_2}$.

\medskip

\noindent
(d) {\it $m\not\perp p_0$, $m\cap \ell_0=\emptyset$ and $m\cap \ell_0^\perp\not=\emptyset$.} We have two subcases.

\medskip

\noindent
(d.1)  {\it $m\subseteq \ell_0^\perp$}. Choose a point $p\in\ell_0$ neither collinear with $p_0$ nor belonging to $H$ and consider the singular plane $X =\langle p,m\rangle_{\cP}$. All  lines of $X$  through $p$ belong to $\langle G \rangle_{\cP_2}$ by Case (c). Moreover $p\not\in X\cap H$.  Hence all lines of $X$ belong to $\langle G\rangle_{\cP_2}$ by Lemma~\ref{lemma0}. In particular, $m\in\langle G\rangle_{\cP_2}$.

\medskip

\noindent
(d.2) {\it $|m\cap \ell_0^\perp|=1$.} Write $p_1:=m^\perp\cap \ell_0$ and consider the singular plane $X:=\langle m,p_1\rangle_{\cP}$. All lines of $X$ through $p_1$ belong to $\langle G\rangle_{\cP_2}$ by Case (c). Suppose firstly that $p_1\in H$. Then $p_1\not \in p_0^\perp \cap X$, since $p_0^\perp \cap \ell_0\not\in H$. So $X$ contains three non-concurrent lines of $\langle G\rangle_{\cP_2}$: two of them are contributed by the pencil of lines through $p_1$; the third one is $p_0^\perp \cap X$, which belongs to $\langle G\rangle_{\cP_2}$ by Case (a). Hence all lines of $X$ belong to $\langle G\rangle_{\cP_2}$ by Lemma~\ref{lemma0}. In particular, $m \in\langle G\rangle_{\cP_2}.$

On the other hand, let $p_1\not\in H$. Then $X$ contains three non-concurrent lines of $\langle G\rangle_{\cP_2}$, namely two lines from the pencil centered at $p_1$ and the line $X\cap H$. Hence $m \in\langle G\rangle_{\cP_2}$, as above.

\medskip

\noindent
(e) {\it $m\cap \ell_0= m\cap \ell_0^\perp =\emptyset$ and $m\not\perp p_0$.} We must consider two subcases.

\medskip

\noindent
(e.1) {\it The point $p_X:= \ell_0^\perp\cap X$ is not in $H\cap p_0^\perp$ for some singular plane $X$ through $m$.}  All lines of $X$ through $p_X$ are in $\langle G\rangle_{\cP_2}$ by Case (d). Moreover $X\cap p_0^\perp\in \langle G\rangle_{\cP_2}$ by Case (a) and $X\cap H\in G$. By Lemma~\ref{lemma0}, all lines of $X$ belong to $\langle G\rangle_{\cP_2}$. In particular, $m\in \langle G\rangle_{\cP_2}.$

\medskip

\noindent
(e.2) {\it The point $p_X:= \ell_0^\perp\cap X$ belongs to $H\cap p_0^\perp$ for every singular plane $X$ through $m$.}  Let $q_0:=\ell_0\cap p_0^{\perp}$ and $p_H:=\ell_0\cap H$. By hypothesis,  $q_0\not\in H$. Let $\ell_0'$ the line of $\cP$ joining $q_0$ and  $q_0^{\perp}\cap m$. If $\ell_0'\not\subseteq p_0^{\perp}$ then $q_0^{\perp}\cap m\not\in p_0^{\perp}$. Put $G' := S_2(H, p_0, \ell'_0)$. Then $m \in \langle G'\rangle_{\cP_2}$ by the case previously examined and $G' \subseteq \langle G\rangle_{\cP_2}$ as $\ell_0'\in\langle G \rangle_{\cP_2}$. Therefore $m\in \langle G\rangle_{\cP_2}$.

On the other hand, let $q_0^{\perp}\cap m\in p_0^{\perp}$ and let $\ell_0'$ be the line joining $p_H$ with $p_H^{\perp}\cap m$. If $\ell_0'\not\subseteq H$, that is $p_H^{\perp}\cap m\not\in H$, then we can argue as in the previous paragraph, thus obtaining that $m\in \langle G\rangle_{\cP_2}$.

So, suppose that $p_{H}^{\perp}\cap m\in H$. The points $p'_H:=p_H^{\perp}\cap m$ and $q_0':=q_0^{\perp}\cap m$ are distinct for, otherwise, $\ell_0^{\perp}\cap m\neq\emptyset$, a contradiction. Let $\ell_H$ be the line through $p_H$ and $p'_H$ and $m_0$ the line through $q_0$ and $q_0'$. We have $\ell_H\subseteq H$ while $m_0\subseteq p_0^{\perp}$. Let $n_0\subseteq p_0^{\perp}$ and  $n_H\subseteq H$ be the lines of $\cP$ joining the point $\ell_H\cap p_0^{\perp}$ with a point of $m_0$ and the point $m_0\cap H$ with a point of $\ell_H$ respectively. There are two possibilities.

\medskip

\noindent
(e.2.1) $n_0\neq n_H$. Let $p$ be a point of $\ell_0$ different from $p_H$ and $q_0$ and let $\ell'_0$ be the line joining $p$ with a point of $n_H$ different from $n_H\cap\ell_H$ and $n_H\cap m_0$. As $n_0\neq n_H$, the point $\ell_0'\cap n_H=\ell_0'\cap H$ is not in $p_0^{\perp}$. So we can define $G' :=  S_2(H, p_0, \ell'_0)$ and we get $G' \subseteq \langle G\rangle_{\cP_2}$.

Suppose that $\ell_0'\cap n_H\in p_H'^{\perp}$ and $\ell'_0\cap p_0^{\perp}\in q_0'^{\perp}$. Then $n_H\perp\ell_H$ and consequently $m_0\perp\ell_0$. Therefore $q_0'\perp\ell_0$, against the assumption $m^{\perp}\cap\ell_0=\emptyset$. So, either $\ell_0'\cap n_H\not\in p_H'^{\perp}$ or $\ell'_0\cap p_0^{\perp}\not\in q_0'^{\perp}$. Hence $m \in \langle G'\rangle_{\cP_2}$ in view of the previous cases. Then $m \in \langle G\rangle_{\cP_2}$ since $G'\subseteq \langle G\rangle_{\cP_2}$.

\medskip

\noindent
(e.2.2) $n_H=n_0$. The line $n := n_H = n_0$ is contained in $p_0^{\perp}\cap H$. Let $X$ be the singular plane $\langle n,p_0\rangle_{\cP}$ and $a:=\ell_0^{\perp}\cap X$. The point $a$ is on the line $X\cap q_0^{\perp}$ joining $p_0$ with $n\cap m_0$.  Furthermore, $a\neq p_0$ since
$\ell_0\not\perp p_0$. Consider the singular plane $Y := \langle a,\ell_0\rangle_{\cP}$, the point  $b:=m^{\perp}\cap Y$ and the singular plane $Z =\langle m,b\rangle_{\cP}$. We have $b\perp \ell_0$. So, every line of $Z$ through $b$ belongs to $\langle G\rangle_{\cP_2}$ by one of the cases from (a) to (d).
If $b\not\in p_0^{\perp}\cap H$, then $Z$ contains a triangle consisting of a line through $b$, a line in $p_0^{\perp}$ and a line in
$H$. Hence all lines of $Z$ are in $\langle G\rangle_{\cP_2}$ by Lemma~\ref{lemma0}. In particular, $m\in\langle G\rangle_{\cP_2}$.

On the other hand, if $b\in p_0^{\perp}\cap H$ then $b$ is the unique point of $H$ on the line $m'$ through $q_0$ and $a$. Since $b\in p_H^{\perp}\cap p_H'^{\perp}$, we have $b\perp\ell_H$. Thus the point $p_H'':=n\cap\ell_H$ is collinear with $b$ as well as with $a$. So $p_H''$ is collinear with $q_0$. Since $p_H''$ is also collinear with $p_H$, then it is collinear with $\ell_0$. Thus $\ell_H\perp\ell_0$ and, consequently, $p_H'\perp\ell_0$. This crashes against the assumption $m^{\perp}\cap\ell_0=\emptyset$. This contradiction concludes the proof.
\end{proof}

\subsubsection{The general case $2 \leq k<n$}\label{k>2}

Let $2\leq k<n$ and let $H$ be a hyperplane of $\cP$. Take a point $p_0$ of $\cP\setminus H$. Then $H\cap p_0^{\perp}$ is a non-degenerate subspace of $\cP$ of rank $n-1$, isomorphic to $\Res(p_0)$. When $k > 2$, let $G_{H,p_0} \subseteq S_{k-2}(H\cap p_0^{\perp})$ be a generating set for the $(k-2)$-Grassmannian $(H\cap p_0^{\perp})_{k-2}$ of $H\cap p_0^{\perp}$. When $k =2$ we put $G_{H,p_0} = \{\emptyset\}$. For any $Z\in G_{H,p_0}$, choose
a $k$-space $\widehat{Z}\in S_k(\cP)$ containing $Z$ and such that $\widehat{Z}\not \subseteq H\cup p_0^\perp$ and $\widehat{Z}\cap H \neq \widehat{Z}\cap p_0^\perp$ (equivalently $\widehat{Z}\cap H\cap p_0^{\perp}= Z$). Put $\widehat{G}_{H,p_0} :=  \{\widehat{Z} : Z\in G_{H,p_0}\}$. In particular, when $k = 2$ the set $\widehat{G}_{H,p_0}$ consists of a single line $\ell_0$ as in Subsection \ref{k=2}. Define
\begin{equation}\label{defgensetk>2}
S_k(H, p_0, \widehat{G}_{H,p_0}) ~ := ~  S_k(H)\cup S_k(p_0)\cup\widehat{G}_{H,p_0} ~ \subseteq ~ S_k(\cP).
\end{equation}
Note that $S_k(H)$ is never empty, since all hyperplanes of $\cP$ have rank either $n$ or $n-1$. Note also that the set $S_2(H,p_0,\ell_0)$ defined in \eqref{defgensetk=2} for $k = 2$ is just a special case of the above.

\begin{lemma}\label{l1}
Suppose that $n \geq 4$. Let $x,y$ be two collinear points of $p_0^{\perp}\cap H$. Then there exists at least one 4-space $X\in S_4(\cP)$ such that
$p_0^{\perp}\cap X\cap H=\langle x,y\rangle_{\cP}$.
\end{lemma}
\begin{proof}
Put $\ell :=\langle x,y\rangle_{\cP}$. We must prove that there exists a $4$-space $X$ of $S_4(\cP)$ such that $X\supset \ell$ and
$p_0^{\perp}\cap X\cap H=\ell$. Equivalently, a line exists in the upper residue $\Res(\ell)^\uparrow$ of $\ell$ which is disjoint from the subspace $S_3(H\cap p_0^{\perp})\cap S_3(\ell)$. The latter is the intersection of two distinct hyperplanes of $\Res(\ell)^\uparrow$, namely $S_3(H)\cap S_3(\ell)$ and $S_3(p_0^\perp)\cap S_3(\ell)$. So, $S_3(H\cap p_0^{\perp})\cap S_3(\ell)$ is a proper subspace but not a hyperplane of $\Res(\ell)^\uparrow$. Hence $\Res(\ell)^\uparrow$ indeed contains a line which misses $S_3(H\cap p_0^{\perp})\cap S_3(\ell)$.
\end{proof}

\begin{lemma}\label{lemma1-gen-l}
With $k \geq 3$, let $Z\in S_{k+1}(\cP)$ and $X_1, X_2\in S_{k-2}(\cP)$ be such that $Z\supset X_1\cup X_2$ (hence $\langle X_1\cup X_2\rangle_{\cP}$ is a singular subspace), $\rank(X_1\cap X_2)=k-3$ and $\rank(\langle X_1\cup X_2\rangle_{\cP})=k-1$. Suppose that all $k$-subspaces of $Z$ through $X_1$ or $X_2$ belong to $\langle S_k(H, p_0, \widehat{G}_{H,p_0})\rangle_{\cP_k}$. Then all $k$-subspaces of $Z$ through $X_1\cap X_2$ belong to $\langle S_k(H, p_0, \widehat{G}_{H,p_0})\rangle_{\cP_k}$.
\end{lemma}
\begin{proof}
Put $G := S_k(H, p_0, \widehat{G}_{H,p_0})$, for short. Also $W:=X_1\cap X_2$ and $U:=\langle X_1,X_2\rangle_{\cP}.$ Let $Y$ be a $k$-space of $Z$ through $W$.
If $Y\supset U$ then $Y\in  \langle G\rangle_{\cP_k}$ by our hypotheses, because $X_1\subset U \subset Y$.

If $Y\not\supset U$ then $\rank(Y\cap U)=k-2$. Hence there exists a line $\ell$ in $Y$ disjoint from $U.$ Put $U':= \langle W, \ell\rangle_{\cP}$. Hence $\rank(U')=k-1$ and $U'\subseteq Y.$ Define $Y_1:=\langle X_1,\ell \rangle_{\cP}$ and $Y_2:=\langle X_2,\ell \rangle_{\cP}$. Since  $X_1\subseteq Y_1$ and $\rank(Y_1)=k$ we have $Y_1\in \langle G\rangle_{\cP_k}.$ Similarly, $Y_2\in \langle G\rangle_{\cP_k}.$  Note that $Y_1\cap Y_2= U'$ and $\rank(\langle Y_1, Y_2\rangle_\cP)=k+1$ (actually, $Z = \langle Y_1, Y_2\rangle_\cP$.) Hence $Y\in  \langle Y_1, Y_2\rangle_{\cP_k} \subseteq \langle G\rangle_{\cP_k}.$
\end{proof}

\begin{prop}\label{proposition1}
The set $S_k(H,p_0, \widehat{G}_{H,p_0})$ spans $\cP_k$, for any $k = 2, 3,..., n-1$.
\end{prop}
\begin{proof}
If $k=2$ the thesis holds by  Proposition~\ref{lemma1}. Suppose $k>2$; so $n\geq 4$. As in the proof of Lemma \ref{lemma1-gen-l}, put $G := S_k(H, p_0, \widehat{G}_{H,p_0})$. Consider the following claim:
\begin{itemize}
\item[$(\ast)$] {\em We have $S_k(Z)\subseteq \langle G\rangle_{\cP_k}$ for every $Z\in S_{k-2}(H\cap p_0^{\perp}).$}
\end{itemize}
Suppose that $(\ast)$ holds true. Take $X\in S_k(\cP)$. We have $\rank(X\cap H\cap p_0^\perp)\geq k-2$, hence $S_{k-2}(X\cap H\cap p_0^\perp) \neq \emptyset$. Let $Y\in S_{k-2}(X\cap H\cap p_0^\perp)$. We have $S_k(Y)\subseteq \langle G\rangle_{\cP_k}$ by claim $(\ast)$. Hence $X\in \langle G\rangle_{\cP_k}$. The equality $\langle G\rangle_{\cP_k} = \cP_k$ follows.

Claim $(\ast)$ remains to be proved. We will prove it by induction on the number $i$ of steps required to reach a given $Z\in S_{k-2}(H\cap p_0^\perp)$ starting from $G_{H,p_0}$. Put $U(0)= G_{H,p_0}$ and recursively define $U(i)\subseteq S_{k-2}(H\cap p_0^\perp)$ as follows:
\[U(i)~ := ~ \{Z \in \langle Z_1, Z_2\rangle_{\cP_{k-2}}\colon Z_1, Z_2\in U(i-1)\}~ \text{ for } i>0.\]
As $G_{H,p_0}$ spans $(H\cap p_0^\perp)_{k-2}$, we have $S_{k-2}(H\cap p_0^\perp) = \cup_{i=0}^\infty U(i)$. So, in order to prove claim $(\ast)$ we only must show that it holds for all elements in $U(i)$ and every $i \geq 0$. We shall prove this by induction on $i$.

Let $i=0$. Take $Z\in U(0)= G_{H,p_0}$. We must show that $S_k(Z)\subseteq \langle G\rangle_{\cP_k}$. By definition of $\widehat{G}_{H,p_0}$, there exists $\widehat{Z}\in \widehat{G}_{H,p_0} \subset G$ containing $Z$ and such that $\widehat{Z}\not \subseteq H\cup p_0^\perp$ and $\widehat{Z}\cap H \neq \widehat{Z}\cap p_0^\perp$. The intersection $\overline{H} := S_{k-1}(H)\cap S_{k-1}(Z)$ is a hyperplane of $\Res(Z)^\uparrow$ while $\bar{p}_0 := \langle Z, p_0\rangle_\cP$ and $\bar{\ell}_0 := \widehat{Z}$ are a point and a line of $\Res(Z)^\uparrow$ respectively.  Consider the set
\[S_k(H, p_0, Z)  ~ = ~ S_2(\overline{H}, \bar{p}_0, \bar{\ell}_0) ~:= ~ (S_k(H)\cap S_k(Z))\cup(S_k(p_0)\cap S_k(Z))\cup\{\widehat{Z}\}.\]
This is the analogous in $\Res(Z)^\uparrow$ of the set $S_2(H,p_0,\ell_0)$ defined in \eqref{defgensetk=2}. Indeed $S_k(H)\cap S_k(Z)$ is the set $S_2(\overline{H})$ of the lines of $\Res(Z)^\uparrow$ contained in the hyperplane $\overline{H}$ of $\Res(Z)^\uparrow$, the intersection $S_k(p_0)\cap S_k(Z)$ is the set $S_2(\bar{p}_0)$ of lines of $\Res(Z)^\uparrow$ through the point $\bar{p}_0$ of $\Res(Z)^\uparrow$ and $\bar{\ell}_0 = \widehat{Z}$ is a line of $\Res(Z)^\uparrow$ neither contained in $\overline{H}$ nor passing through $\bar{p}_0$ and such that $\bar{\ell}_0\cap\overline{H} \neq \bar{\ell_0}\cap \bar{p}_0^\perp$.

By Proposition \ref{lemma1} the set $S_k(H, p_0, Z) = S_2(\overline{H}, \bar{p}_0, \bar{\ell}_0)$ spans the $2$-Grassmannian of $\Res(Z)^\uparrow$. However, the latter is just the intersection of $\Res(Z)^\uparrow$ with the $k$-Grassmannian $\cP_k$ of $\cP$. Therefore $S_k(H, p_0, Z)$ spans $S_k(Z)$ in $\cP_k$, as claimed in $(\ast)$. So, $(\ast)$ is proved for $U(0)$.

Let now $i>0$ and assume that $(\ast)$ holds for $U(i-1)$. Take $Z\in U(i).$ By definition of $U(i)$ there exist $Z_1, Z_2\in U(i-1)$ with
$Z_1\cap Z_2\subset Z \subset\langle Z_1, Z_2\rangle_{\cP}$, where $\rank(Z_1\cap Z_2)=k-3$ and $\langle Z_1, Z_2\rangle_{\cP}$ is a singular subspace of rank $k-1$. Also, $Z_1, Z_2\in S_{k-2}(p_0^\perp\cap H)$. Put  $W:= Z_1\cap Z_2$ and $U:=\langle Z_1, Z_2\rangle_{\cP}.$
By the inductive hypothesis, all $k$-singular spaces through $Z_1$ or $Z_2$ belong to $\langle G\rangle_{\cP_k}$.  By Lemma~\ref{l1} in $\Res(W)^\uparrow$ and with $Z_1$ and $Z_2$ regarded as two collinear points of $\Res(W)^\uparrow$ both contained in the subspace $S_{k-2}(H\cap P_0^\perp)\cap S_{k-2}(W)$ of $\Res(W)^\uparrow$, there exists $Y\in S_{k+1}(W)$ such that $Y\cap p_0^\perp\cap H=U.$ Note that $Z\subset Y$, as $Z\subset U \subset Y$.

By Lemma~\ref{lemma1-gen-l} all $k$-spaces in $Y$ through $W$ are in $\langle G\rangle_{\cP_k}.$ Take a $k$-subspace $\widehat{Z}$ in $Y$ through $Z$ with the property that $\widehat{Z}\cap p_0^\perp\cap H= Z$. Since $W\subset Z\subset \widehat{Z}\subset Y$, we have that $\widehat{Z}\in \langle G\rangle_{\cP_k}.$ Apply now Proposition~\ref{lemma1} in $\Res(Z)^\uparrow$ with $\bar{\ell}_0 := \widehat{Z}$, $\overline{H} := S_k(Z)\cap S_k(H)$ and $\bar{p}_0 := \langle Z, p_0\rangle_\cP$ in the roles of $\ell_0$, $H$ and $p_0$ respectively. We obtain that all lines of $\Res(Z)^\uparrow$ belong to the span of $S_2(\overline{H}, \bar{p}_0, \bar{\ell}_0)$ in the $2$-Grassmannian of $\Res(Z)^\uparrow$. Equivalently, $S_k(Z) =  \langle S_k(H, p_0, Z) \rangle_{\cP_k} \subseteq \langle G\rangle_{\cP_k}$.  Claim $(\ast)$ holds for $U(i)$. The inductive step of the proof is complete. By induction, $(\ast)$ holds true.
\end{proof}

Suppose that $H$ is singular and let $q$ be its deep point. So, $H = q^\perp$ and $H\cap p_0^\perp = \{q,p_0\}^\perp$. As above, given a generating set $G_{q,p_0} := G_{H,p_0}$ of the $(k-2)$-Grassmannian of $\{q,p_0\}^\perp$ (with  $G_{q,p_0} = \{\emptyset\}$ when $k = 2$), for every $Z\in G_{q,p_0}$ choose $\widehat{Z}\in S_k(Z)$ such that $\widehat{Z}\cap\{q,p_0\}^\perp = Z$ and put $\widehat{G}_{q,p_0} := \{\widehat{Z} : Z\in G_{q,p_0}\}$. Next define
\begin{equation}
S_k(q,p_0,\widehat{G}_{q,p_0})~:= ~ S_k(q)\cup S_k(\{q,p_0\}^{\perp})\cup S_k(p_0)\cup\widehat{G}_{q,p_0} ~ \subseteq ~ \cP_k.
\end{equation}

\begin{corollary}\label{cor2}
The set $S_k(q,p_0,\widehat{G}_{q,p_0})$ spans  $\cP_k.$
\end{corollary}
\begin{proof}
We shall prove that $S_k(q)\cup S_k(\{q, p_0\}^\perp)$ spans $S_k(q^\perp)$. Then the conclusion follows from Proposition \ref{proposition1}.

Let $X\in S_k(q^\perp)$ such that $q\not\in X\not\subseteq q_0^\perp$. Put $Y := \langle X, q\rangle_\cP$ and $Z := X\cap p_0^\perp$. Then $Y \in S_{k+1}(q^\perp)$ and $Z\in S_{k-1}(q^\perp)$. The subspaces $Z$ and $Y$ uniquely determine a line $L$ of $\cP_k$ contained in $S_k(q^\perp)$. Clearly $X\in L$. The $k$-subspaces $\langle Z, q\rangle_\cP$ and $Y\cap p_0^\perp$ belong to $L$. Hence $L\subseteq \langle S_k(q)\cup S_k(\{q, p_0\}^\perp)\rangle_{\cP_k}$. In particular, $X\subseteq \langle S_k(q)\cup S_k(\{q, p_0\}^\perp)\rangle_{\cP_k}$.
\end{proof}

In order to apply Proposition \ref{proposition1}, apart from knowing $\widehat{G}_{H,p_0}$, we need to determine
a generating set for the geometry $\cP_k(H)$ induced by $\cP_k$ on $S_k(H)$ and a generating set for the $(k-1)$-Grassmannian of $\Res(p_0)$ of $\mathrm{\Res}(p_0)$. In view of Corollary \ref{cor2}, in the special case $H=q^{\perp}$ a generating set for $S_k(q)$ and one for $\cP_k(\{p_0,q\}^\perp)$ are enough to generate $\cP_k(H)$.

We warn that when $k = n-1$ and $H$ is non-degenerate of rank $n-1$ the geometry $\cP_{n-1}(H)$ admits no lines; consequently, $S_{n-1}(H)$ is the unique generating set for $\cP_{n-1}(H)$. In this case Proposition \ref{proposition1} is of no use. Similarly, Corollary  \ref{cor2} is useless when $k = n-1$. Indeed $\rank(\{q,p_0\}^\perp) = n-1$, whence $\cP_{n-1}(\{q,p_0\}^\perp)$ has no lines.

\section{Hermitian Grassmannians}\label{Herm}

Throughout this section $\cH$ is a non-degenerate Hermitian polar space of finite rank $n \geq 2$ and defect $d$ and $\cH_k$ is its $k$-Grassmannian, for $1\leq k\leq n$. We recall that the underlying field $\FF$ of $\cH$ is a separable quadratic extension of a subfield $\FF_0$. We denote by $\sigma$ the unique non trivial element of $\mathrm{Gal}(\FF:\FF_0)$.

By Corollary \ref{3&4}, we can regard $\cH$ as the polar space associated to a non-degenerate Hermitian form $h:V\times V\rightarrow \FF$ of Witt index $n$ and defect $\mathrm{def}(h) = \mathrm{def}(\cH) = d$, for a $(2n+d)$-dimensional $\FF$-vector space $V$. In view of the decomposition \eqref{decomposition}, we can assume to have chosen a basis $B = (e_1, e_2,... e_{2n}, e_{2n+1},...)$ of $V$ such that $h$ can be expressed as follows with respect to $B$:
\begin{equation}\label{Herm - h}
h(\sum_{i > 0}e_ix_i, \sum_{i> 0}e_iy_i) =
    \sum_{i=i}^{n}\left(x_{2i-1}^{\sigma}y_{2i}+x_{2i}^{\sigma}y_{2i-1}\right)+ h_0(\sum_{j > 2n}e_jx_j, \sum_{j > 2n}e_jy_j)
\end{equation}
with $h_0$ anisotropic. Clearly, $h_0$ is the form induced by $h$ on $V_0\times V_0$, where $V_0 := \{e_1,..., e_{2n}\}^\perp = \langle e_{2n+1}, e_{2n+2},...\rangle_V$ is the orthogonal of $\langle e_1,..., e_{2n}\rangle_V$ with respect to $h$.

Before to discuss the generating rank of $\cH_k$, we state the following elementary lemma.

\begin{lemma}\label{easy}
Let $d > 0$. Then all non-degenerate hyperplanes of $\cH$ have rank $n$ and defect $d-1$ ($ = d$ if $d$ is infinite).
\end{lemma}
\begin{proof}
Let $K$ be a non-degenerate hyperplane of $\cH$. In $\PG(V)$ we see that $K = p^\perp\cap \cH$ where $p := K^\perp$ is a point of $\PG(V)$, non-singular for $h$. Conversely, for every point $p\in \PG(V)\setminus \cH$, the subspace $K := p^\perp\cap\cH$ is a non-singular hyperplane of $\cH$. With no loss, we can assume that $p = [v]$ where $v = e_1+te_2$ for $t\in \FF$ such that $t+t^\sigma \neq 0$. Put
\[u ~ := ~ e_1 - t^\sigma e_2 + \frac{t+t^\sigma}{\kappa}e_{2n+1}\]
where $\kappa := h(e_{2n+1},e_{2n+1}) = h_0(e_{2n+1},e_{2n+1}) \in \FF_0\setminus \{0\}$. The vector $u$ is isotropic for $h$ and orthogonal to each of $e_3, e_4,..., e_{2n}$. Hence $e_3, e_5,..., e_{2n-1}, u$ span an $n$-dimensional subspace $X$ of $V$, totally isotropic for $h$ and orthogonal to $p$. So, $[X]\subseteq K$.
It follows that $\rank(K) = n$. As $\dim(v^\perp) = 2n+d-1$, the form induced by $h$ on $v^\perp$ has defect $d-1$. Hence $\mathrm{def}(K) = d-1$.
\end{proof}

\subsection{The case $k=n$}\label{herm dps}

Let $F = \{p_1,..., p_n, p'_1,..., p'_n\}$ be a frame of $\cH$, now regarded as a set of $2n$-points rather than a partition $\{\{p_1,..., p_n\},\{ p'_1,..., p'_n\}\}$ of this set in two sets of $n$ mutually collinear points. We recall that the collinearity graph of $\cH$ induces on $F$ a complete $n$-partite graph $(F,\perp)$ with classes of size $2$. We can assume that $\{p_1, p'_1\},..., \{p_n, p'_n\}$ are the $n$ classes of this partition. The graph $(F,\perp)$ contains just $2^n$ maximal cliques; each of them picks just one point from each of the pairs  $\{p_1, p'_1\},..., \{p_n, p'_n\}$. Let $F_n\subset S_n(\cH)$ be the collection of the $n$-subspaces of $\cH$ spanned by the maximal cliques of $(F,\perp)$. So, $|F_n|=2^n$. Following Pankov \cite{Pank1}, we call $F_n$ an {\em apartment} of $\cH_n$.

The next theorem generalizes a result of Cooperstein and Shult \cite[Section 3]{CS01}. Explicitly, under the hypothesis that $d = 1$ and $\FF$ is finite, Cooperstein and Shult prove that $\gr(\cH_n) \leq 2^n$. We obtain the same conclusion, but with no hypotheses on $\FF$ and allowing any value for $d$ (even an infinite one), except $0$.

\begin{theorem}\label{k=n; 1}
Let $d > 0$. Then $F_n$ spans $\cH_n$.
\end{theorem}
\begin{proof}
Consider the following property:
\begin{itemize}
\item[(a)] $S_n(p_i)\cup S_n(p'_i)\subseteq \langle F_n\rangle_{\cF_n}$, for every $i = 1, 2,..., n$.
\end{itemize}
We firstly prove the following:
\begin{itemize}
\item[(b)]  If property (a) holds then $F_n$ spans $\cH_n$.
\end{itemize}
Assume (a). Let $X\in \cH_n\setminus(S_n(p_i)\cup S_n(p'_i))$ and suppose that $X\cap p_i^\perp = X\cap p_i'^\perp =: Y$, say. Then $X$ belong to the line of $\cH_b$ spanned by the points $X_i := \langle Y, p_i\rangle_\cP$ and $X'_i = \langle Y, p'_i\rangle_\cP$ of $F_n$. So,
\begin{itemize}
\item[(c)] $\langle F_n\rangle_{\cH_n}$ contains all $X\in S_n(\cH)$ such that $X\cap p_i^\perp = X\cap p_i'^\perp$ for some $i = 1, 2,..., n$.
\end{itemize}
Consider the hyperbolic line $\ell = \{p_1, p'_1\}^{\perp\perp}$. Regarded $\cH$ as embedded in $\PG(V)$, the hyperbolic line $\ell$ is properly contained in a line $L$ of $\PG(V)$. Take $p\in L\setminus \ell$ and consider the hyperplane $p^\perp$, where $\perp$ is the orthogonality relation relative to a form $h:V\times V\rightarrow \FF$ which defines $\cH$. Then $K := p^\perp\cap \cH$ is a non-singular hyperplane of $\cH$ and, since $d > 0$, it has rank $n$ by Lemma \ref{easy}. Moreover, $K\cap p_1^\perp = K\cap p_1'^\perp$. Hence $X\cap p_1^\perp = X\cap p_1'^\perp$ for every $X\in S_n(K)$. Therefore $S_n(K) \subseteq \langle F_n\rangle_{\cH_n}$ by claim (c). However $S_n(K)$ spans $\cH_n$ by Proposition \ref{k=n; 0}. Consequently $\langle F_n\rangle_{\cH_n} = \cH_n$. Claim (b) is proved.

Claim (a) remains to be proved. We shall prove it by induction on $n$. Let $n=2$. Then $\Res(p_i)$ and $\Res(p'_i)$ are lines of $\cH_2$. Each of them belongs to just two elements of $F_2$. So, the elements of $F_2$ span $S_2(p_i)$ as well as $S_2(p'_i)$. Let now $n > 2$. With no loss, we can assume that $p_i = [e_{2i-1}]$ and $p'_i = [e_{2i}]$ for $i = 1, 2,... n$. The subspace $V':=\{e_1,e_2\}^{\perp}$ of $V$ has dimension $2(n-1)+d$ and the form $h$ induces on $V'$ a non-degenerate Hermitian form $h'$ of Witt index $n-1$ and anisotropic defect $d$. So, the polar space $\cH'$ associated with $h'$ has rank $n-1$ and defect $d$. Moreover $F' := \{p_2,..., p_n, p'_2,...., p'_n\}$ is a frame of $\cH'$. Let $\cH'_{n-1}$ be the dual of the polar space $\cH'$. The apartment $F'_{n-1}$ of $\cH'_{n-1}$ associated to $F'$ consists of the $2^{n-1}$ intersections $[V']\cap X$ for $X\in F_n$. (Observe that the equation $X\cap [V'] = Y\cap [V']$ defines an equivalence relation on $F_n$ with $2^{n-1}$ classes, all of which have size $2$.) By the inductive hypothesis, (a) holds for $\cH'_{n-1}$ and $F'_{n-1}$. Hence $F'_{n-1}$ spans $\cH'_{n-1}$, by claim (b).

However $S_n(p_1) = \{\langle p_1, X'\rangle_{\cH} : X'\in S_{n-1}(\cH')\}$ and, for every line $\Res(Y')^\uparrow$ of $\cH'_{n-1}$, the residue $\Res(Y)^\uparrow$ of $Y := \langle p_1, Y'\rangle_{\cH}$ is a line of $\cH_n$. Moreover, the $n$-spaces of $F_n$ containing $p_1$ are just the spans $\langle p_1, X'\rangle_\cH$ for $X'\in F'_{n-1}$. It follows that $S_n(p_1)\subseteq \langle F_n\rangle_{\cF_n}$. Similarly, $S_n(p'_1) \subseteq \langle F_n\rangle_{\cF_n}$. Clearly, the same argument works for $p_i$ and $p'_i$ for every $i = 2, 3,..., n$. Claim (a) is proved.
\end{proof}

The next corollary contains part (2) of Theorem \ref{MT1}.

\begin{corollary}\label{csmall}
Let $d>0$. Then $\gr(\cH_n)\leq 2^n.$ Moreover, if $n=2$ then $\gr(\cH_2)=4$.
 \end{corollary}
\begin{proof}
The general claim immediately follows from Theorem \ref{k=n; 1} and the fact that $|F_n|=2^n$. When $n=2$, the dual polar space $\cH_2$ is a generalized quadrangle; no generalized quadrangle can be generated by less than $4$ points.
\end{proof}

\subsection{The case $k<n$}\label{herm k<n}

The next theorem contains part (1) of Theorem \ref{MT1}. It is based on a result Blok and Cooperstein \cite{BC2012}, where it is stated that if $\FF\neq \FF_4$ and $d = 0$ then $\gr(\cH_k) = {{2n}\choose k}$. We shall prove that the same holds for any value of the defect $d$, modulo replacing $2n$ with $2n+d$ in the binomial coefficient.

\begin{theorem}\label{gr herm grass non-deg}
Suppose that $\FF\not=\FF_4.$ Then $\gr(\cH_k)={{2n+d}\choose k}$ ($ = d$ when $d$ is infinite) for every $k = 1, 2,..., n-1$.
\end{theorem}
\begin{proof}
We only must prove that
\begin{itemize}
\item[$(\ast)$]  $\gr(\cH_k) \leq {{2n+d}\choose k}$ for every $k = 1, 2,..., n-1$.
\end{itemize}
Indeed, since $\cH_k$ admits an embedding of dimension ${{2n+d}\choose k}$ (namely the Pl\"{u}cker embedding \cite{ILP}), we have $\er(\cH_k) \geq {{2n+d}\choose k}$. By combining this inequality with $(\ast)$ we get the thesis of the theorem.

Note firstly that $(\ast)$ holds for $k = 1$ by Proposition \ref{t:k=1}. (Recall that $\mathfrak{N}(\cH)$ admits maximal well ordered chains even if $d$ is infinite, by Theorem \ref{4}.)

We proceed by induction on $n$. When $n = 2$ then necessarily $k = 1$. Hence $(\ast)$, as noticed above. Let $n>2$ and let $(K^{(\delta)})_{\delta\leq\omega}$ be a maximal well ordered chain of $\mathfrak{N}(\cH)$, where $|\omega| = d$. Recall that $K^{(\delta)}$ is a non-degenerate polar space of rank $n$. We have $\mathrm{def}(K^{(\delta)}) = |\delta|$ by Theorems \ref{3} and \ref{4}. Moreover, $K^{(\delta)}$ is Hermitian, associated to the form induced by $h$ on the subspace of $V$ corresponding to the subspace of $\PG(V)$ spanned by $K^{(\delta)}$ (see Lemma \ref{2}).

We shall prove by induction on $\delta$ that $(\ast)$ holds for $K^{(\delta)}$ for every $\delta \leq\omega$. It holds for $K^{(0)}$ by \cite{BC2012}. Suppose it holds for $K := K^{(\delta)}$ and consider $K' := K^{(\delta+1)}$. Since $(\ast)$ holds when $k = 1$, we can assume $k > 1$. By Lemma \ref{3 new} and Corollary \ref{2 bis}, the subspace $K'$ is a hyperplane of $K$. Let $p_0\in K\setminus K'$ and let $\Res(p_0)$ be the residue of $p_0$ in $K'$. By Proposition \ref{proposition1}, the set $S_k(K,p_0, \widehat{G}_{K,p_0})$ spans the $k$-Grassmannian $K'_k$ of $K'$. The $k$-Grassmannian $K_k$ of $K$ admits a generating set $G_K$ of size at most ${{2n+|\delta|}\choose k}$ by the inductive hypothesis on $\delta$. We have $K\cap p_0^\perp\cong \Res(p_0)$ and $\Res(p_0)$ has rank $n-1$ but the same defect $|\delta+1|$ as $K'$. Therefore $(k-2)$-Grassmannian $(K\cap p_0^\perp)_{k-2}$ of $K\cap p_0^\perp$ admits a generating set $G_{K,p_0}$ of size
\[|G_{K,p_0}| ~ \leq ~ {{2(n-1)+|\delta+1|}\choose{k-2}} ~ = ~ {{2n-1+|\delta|}\choose{k-2}}\]
by the inductive hypothesis on $n$. Accordingly, $|\widehat{G}_{K,p_0}| \leq {{2n-1+|\delta|}\choose{k-2}}$. (This also holds when $k = 2$: indeed in this case $G_{K,p_0} = \{\emptyset\}$ has size $1 = {{2n-1+|\delta|}\choose 0}$ and $\widehat{G}_{K,p_0}$ consists of a single line.) Finally, $S_k(p_0)$ is the point-set of the $(k-1)$-Grassmannian of $\Res(p_0)$. By the inductive hypothesis on $n$ we can generate it by a subset $G_{p_0} \subseteq S_k(p_0)$ of size
\[|G_{p_0}| ~ \leq ~ {{2(n-1)+|\delta+1|}\choose{k-1}} ~= ~ {{2n-1+|\delta|}\choose{k-1}}.\]
To sum up, the $k$-Grassmannian $K'_k$ of $K'$ can be generated by a set $G_K\cup\widehat{G}_{K,p_0}\cup G_{p_0}$ of size
\begin{multline*} |G_K\cup\widehat{G}_{K,p_0}\cup G_{p_0}| \leq {{2n+|\delta|}\choose k} + {{2n-1+|\delta|}\choose {k-1}} + {{2n-1+|\delta|}\choose {k-2}}=\\ = {{2n+|\delta|}\choose k} + {{2n+|\delta|}\choose {k-1}}= {{2n+|\delta|+1}\choose {k}} = {{2n+|\delta+1|}\choose {k}}
\end{multline*}
So $(\ast)$ hold for $K'$ too. (Needless to say, when $\delta$ is infinite the above computations amount to the following triviality: $|\delta|+|\delta|+|\delta| = |\delta| + |\delta| = |\delta|$.) In order to finish the proof we must consider the case of $K^{(\delta)}$ with $\delta$ a limit ordinal. In this case $\delta$ is infinite and we have $S_k(K^{(\delta)}) = \bigcup_{\gamma < \delta}S_k(K^{(\gamma)})$. By the inductive hypothesis, the $k$-Grassmannian $K^{(\gamma)}_k$ can be spanned by a subset $G_\gamma$ of size $|G_\gamma| \leq {{2n+|\gamma|}\choose k}$. Clearly, $G_\delta := \bigcup_{\gamma < \delta}G_\gamma$ spans $K^{(\delta)}_k$. However, ${{2n+|\gamma|}\choose k} \leq |\delta|$ for every $\gamma < \delta$, because $\delta$ is infinite. It follows that $|G_\delta|\leq |\delta| = {{2n+|\delta|}\choose k}$. The proof is complete.
\end{proof}

\section{Orthogonal  Grassmannians}\label{OG}

In this section $\cQ$ is a non-degenerate orthogonal polar space of finite rank $n \geq 2$ and defect $d$ and $\cQ_k$ is its $k$-Grassmannian, for $1\leq k\leq n$. We will only consider the cases $k = n$ and $k = 2$. When $k = 1$ we know that $\gr(\cQ_1) = \er(\cQ_1) = 2n+d$ by Corollary \ref{t:k=1 symplectic even}; we have nothing to add to that. Regretfully, we cannot offer any interesting new results on the case $2 < k < n$.

Henceforth $\FF$ stands for the underlying field of $\cQ$. Likewise in the Hermitian case, we can regard $\cQ$ as the polar space associated to a non-degenerate quadratic form $q:V\rightarrow\FF$ of Witt index $n$ and defect $\mathrm{def}(q) = \mathrm{def}(\cQ) = d$, for a $(2n+d)$-dimensional $\FF$-vector space $V$. We can also assume to have chosen a basis $B = (e_1, e_2,... e_{2n}, e_{2n+1},...)$ of $V$ such that $q$ can be expressed as follows with respect to $B$:
\begin{equation}\label{forma quadratica}
 q(\sum_{i > 0}e_ix_i)  =  \sum_{i=1}^nx_{2i-1}x_{2i}+  q_0(\sum_{j> 2n}e_jx_j)
\end{equation}
with $q_0$ anisotropic, namely $q_0(v) \neq 0$ for every non-zero vector $v \in V_0 := \langle e_{2n+1}, e_{2n+2},...\rangle_V$. It goes without saying that $q_0$ is the form induced by $q$ on $V_0$. Clearly, $V_0 := \{e_1,..., e_{2n}\}^\perp$ is the orthogonal of $\langle e_1,..., e_{2n}\rangle_V$ with respect to (the bilinearization of) $q$.

\subsection{Two sub-defects when $\ch(\FF) = 2$}

Let $f:V\times V\rightarrow \FF$ be the bilinearization of $q$. We recall that when $\ch(\FF)\neq 2$ the form $f$ is non-degenerate, with the same Witt index and defect as $q$, and defines the same polar space as $q$, namely $\cQ$.

On the other hand, let $\ch(\FF) = 2$. Then $f$ is a possibly degenerate alternating form and defines a (possibly degenerate) polar space which properly contains $\cQ$ as a subspace. Let $R = \Rad(f)$ be the radical of $f$. Then $R \subseteq V_0$. Let $V'_0$ be a complement of $R$ in $V_0$ and put
\[d_1 := \dim(R), \hspace{8 mm} d_2 := \dim(V'_0) = \mathrm{cod}_{V_0}(R) = \mathrm{cod}_V(\langle\{ e_1,..., e_{2n}\}\cup R\rangle_V).\]
So, $d = d_1 + d_2$. Clearly, the numbers $d_1$ and $d_2$ do not depend on the particular choice of the basis $B$; they only depend on the form $q$. Rather, they only depend on the equivalence-and-proportionality class of $q$. So, as $q$ is uniquely determined by $\cQ$ modulo equivalence and proportionality, these numbers only depend on $\cQ$. Borrowing two words from a terminology popular among finite geometers, we call $d_1$ and $d_2$ the {\em parabolic sub-defect} and the {\em elliptic sub-defect} of $\cQ$ respectively, denoting them by $\mathrm{def}_1(\cQ)$ and $\mathrm{def}_2(\cQ)$.

Note that the form induced by $q_0$ on $R$ is necessarily expressed as a sum of monomials of degree $2$. On the other hand, suppose that $d_2 > 0$ and let $v\in V'_0\setminus\{0\}$. Then $V_0'\setminus v^\perp\neq \emptyset$. Indeed if otherwise then $v^\perp = V$, namely $v\in R$; impossible, since $V'_0\cap R = \{0\}$. Therefore, if $V'_0 = \oplus_{j\in J}V'_{0,j}$ is a decomposition of $V'_0$ as a sum of mutually orthogonal subspaces then $\dim(V'_{0,j}) \geq 2$ for every $j\in J$. Consequently, $d_2\geq 2$.

\begin{lemma}\label{V'_0 decomposition}
For every positive integer $m$ such that $2m \leq d_2$, we can always choose the summands $V'_{0,j}$ of an orthogonal decomposition $V'_0 = \oplus_{j\in J}V'_{0,j}$ in such a way that $m$ of them are $2$-dimensional. In particular if $d_2$ is finite then it is even and $V'_0$ decomposes as the direct sum of $d_2/2$ mutually orthogonal $2$-dimensional subspaces.
\end{lemma}
\begin{proof}
Pick $v_1\in V'_0\setminus \{0\}$. As remarked above, there exists $w_1\in V'_0$ such that $v_1\not\perp w_1$. Put $V'_{0,1} = \langle v_1, w_2\rangle_V$. By the above, $V'_0\setminus x^\perp \neq \emptyset$ for every $x\in V'_{0,1}$. Hence $W_1 := V_{0,1}'^\perp\cap V'_0$ is a complement of $V'_{0,1}$ in $V'_0$. If $m > 1$ we can repeat the previous construction in $W_1$ instead of $V'_0$, thus obtaining a second $2$-subspace $V'_{0,2}$ orthogonal to $V'_{0,1}$. Again, $W_2 := V_{0,2}'^\perp\cap W_1$ is a complement of $V'_{0,2}$ in $W_1$. If $m > 2$ we switch to $W_2$ and go on in the same way.
\end{proof}

Recalling that $q_0$ is anisotropic, by Lemma \ref{V'_0 decomposition} we immediately obtain that, if $\FF$ is perfect, then $(0,0)$, $(1, 0)$, $(0,2)$ are the only possibilities for the pair $(d_1,d_2)$. They correspond to the cases usually called {\em hyperbolic}, {\em parabolic} and {\em elliptic} respectively.

\begin{prob}
Find a characterization of $d_1 = \mathrm{def}_1(\cQ)$ and $d_2 = \mathrm{def}_2(\cQ)$ in the same vein of Definition \label{defect} and Section \ref{anisotropic defect}, without any explicit reference to the form $q$.
\end{prob}

\subsection{The case $k=n$}\label{orthog dps}

The next theorem embodies Theorem \ref{MT2} of the Introduction.

\begin{theorem}\label{orth-dps}
Let $\mathrm{def}(\cQ) = d > 0$. When $\ch(\FF) = 2$ assume furthermore that $\mathrm{def}_2(\cQ) > 0$. Then $\gr(\cQ_n) \leq 2^n$.
\end{theorem}
\begin{proof}
Let $d > 0$. By Theorems \ref{3} and \ref{4}, the Maximal Chain condition (MC) holds in $\cQ$. Suppose firstly that $\ch(\FF) \neq 2$ and let $K$ be a nice subspace of $\cQ$ of defect $\mathrm{def}(K) = 1$ (clearly, $K = \cQ$ if $d = 1$).  For instance, $K$ can be the polar space defined by $q$ on $\langle e_1,..., e_{2n}, v\rangle_V$, for any $v\in V_0\setminus\{0\}$. Then $\gr(K_n) = 2^n$ by \cite{BB98, CS97}. Therefore $\gr(\cQ_n) \leq 2^n$ by Corollary \ref{clkn bis}.

On the other hand, let $\ch(\FF) = 2$ and $d_2 > 0$. By Lemma \ref{V'_0 decomposition} we can choose a nice subspace $K$ of $\cQ$ with $\mathrm{def}_1(K) = 0$ and $\mathrm{def}_2(K) = 2$. Then $\gr(K_n) = 2^n$ by \cite{CS01, B11}. Hence $\gr(\cQ_n) \leq 2^n$ by Corollary \ref{clkn bis}.
\end{proof}

\subsection{Restriction to a subfield of $\FF$}\label{real sec}

Before to turn to the case $k = 2 < n$ we need to fix some terminology and notation, to be exploited in the discussion of that case.

Recall that $V(N,\FF)$ is the $\FF$-vector space of almost everywhere null mappings from a given set $I$ of cardinality $N$ to $\FF$, a mapping $v:I\rightarrow\FF$ being almost everywhere null if $v(i) = 0$ for all but finitely many choices of $i\in I$. Given a subfield $\FF_0$ of $\FF$, the $\FF_0$-linear combinations of the vectors of $B$ form an $\FF_0$-vector space $V(N,\FF_0) \subseteq V(N,\FF)$.

In the setting adopted for orthogonal polar spaces at the beginning of this section, there is no real loss in assuming that $V = V(N,\FF)$ where $N := 2n+d$ and that $B$ is the natural basis of $V(N,\FF)$. Suppose that, modulo equivalence and proportionality if necessary, we have chosen the form $q$ in such a way that it admits a description as in \eqref{forma quadratica} and $q(v)\in \FF_0$ for every vector $v\in V(N,\FF_0)$. Then $q$ induces a quadratic form $q_{|\FF_0}$ on $V(N,\FF_0)$ with the same Witt index $n$ and the same defect $d$ as $q$. Moreover, when $\ch(\FF) = 2$ we also have $\mathrm{def}_1(q_{|\FF_0}) = \mathrm{def}_1(q) = d_1$ and $\mathrm{def}_2(q_{|\FF_0}) = \mathrm{def}_2(q) = d_2$.

As in the second part of Subsection \ref{Main Sec}, we denote by $\cQ(\FF)$ respectively $\cQ_k(\FF)$ the polar space and the $k$-Grassmannian associated to $q$ and by $\cQ(\FF_0)$ respectively $\cQ_k(\FF_0)$ the corresponding geometries arising from $q_{|\FF_0}$. The following is a completion of Definition \ref{subgeometry}.

\begin{definition}\label{real}
We say that a subspace $X$ of $\PG(N-1,\FF)$ is \emph{$\FF_0$-rational} or \emph{defined over $\FF_0$} if it admits a basis consisting of (points of $\PG(N-1,\FF)$ represented by) vectors of $V(N,\FF_0)$. Accordingly, a subset of $S_k(\cQ(\FF))$ is $\FF_0$-{\em rational} if all of its elements are $\FF_0$-rational. We say that an element (a subset) of $S_k(\cQ(\FF))$ is {\em generated over} $\FF_0$ (also {\it $\FF_0$--generated}) if it belongs (is contained) in $\langle G\rangle_{\cQ_k(\FF)}$ for an $\FF_0$-rational subset of $S_k(\cQ(\FF))$. In particular, $\cQ_k(\FF)$ is $\FF_0$--{\em generated} if it admits an $\FF_0$-rational generating set (compare Definition \ref{subgeometry}).
\end{definition}

Regarded $\PG(N-1,\FF_0)$ as a subgeometry of $\PG(N-1,\FF)$, every subspace $X$ of $\PG(N-1,\FF_0)$ is a subset of $\PG(N-1,\FF_0)$. So, we can consider its span in $\PG(N-1,\FF)$, which we will denote by $\langle X\rangle_\FF$ and call the $\FF$-{\em extension} of $X$. Clearly,  $\dim(\langle X\rangle_\FF) = \dim(X)$ and $\langle X\rangle_\FF$ is $\FF_0$-rational. Conversely, if $X$ is a subspace of $\PG(N-1,\FF)$ then $X_{|\FF_0} := X\cap\PG(N-1,\FF_0)$ is a subspace of $\PG(N-1,\FF_0)$, which we call the $\FF_0$-{\em restriction} of $X$; we have $\dim(X) = \dim(X_{|\FF_0})$ if and only if $X$ is the $\FF$-extension of $X_{|\FF_0}$ if and only if $X$ is $\FF_0$-rational.

In view of the above, when no dangerous confusion arises, we can freely regard subspaces of $\PG(N-1,\FF_0)$ as subspaces of $\PG(N-1,\FF)$ and conversely $\FF_0$-rational subspaces of $\PG(N-1,\FF)$ as subspaces of $\PG(N-1,\FF_0)$, thus implicitly identifying a subspace of $\PG(N-1,\FF_0)$ with its $\FF$-extensions and an $\FF_0$-rational subspace of $\PG(N-1,\FF)$ with their $\FF_0$-restrictions. In this way we can regard $\cQ_k(\FF_0)$ as a subgeometry of $\cQ_k(\FF)$, as we have done in Subsection \ref{Main Sec}; in this free setting, $\cQ_k(\FF)$ is $\FF_0$-generated if and only if it is spanned by a subset of $\cQ_k(\FF_0)$ if and only if $\cQ_k(\FF) = \langle \cQ_k(\FF_0)\rangle_{\cQ_k(\FF)}$.

\subsection{The case $k=2<n$}\label{orth k<n}

Cooperstein \cite{C98b} has proved for $k=2 < n$ that if $\FF$ is a prime finite field then $\gr(\cQ_2)={N\choose 2}$ for $0\leq d\leq 2$.
Cardinali and Pasini~\cite{IP13} have determined the embedding rank of $\cQ_k$ for $k=2<n$ and $k=3<n$ under the hypothesis $d \leq 1$, showing that $\er(\cQ_2)={N\choose 2}$, while $\er(\cQ_3)={N\choose 3}$, if $\FF$ is a perfect field of positive characteristic or a number field (and $\FF\not\cong \FF_2$ for $k=3$). This provides a lower bound on the generating rank.

In view of \cite{C98b}, it is natural to ask whether $\cQ_k(\FF)$ can be generated over a given subfield $\FF_0$ of $\FF$; in particular, this would help to determine $\gr(\cQ_k(\FF))$ for $\FF$ a non-prime finite field. The previous problem is investigated in \cite[Theorem 3.1 and Proposition 3.2]{BPa01}. For the case $k=2$, Blok and Pasini~\cite[Corollary 5.4]{BPa01} have proved that when $\FF$ is a finite
extension of the field $\FF_0$ by means of elements $\alpha_1,\ldots,\alpha_t$ then $\gr(\cQ_2(\FF))\leq\gr(\cQ_2(\FF_0))+t$; see~\cite[Corollary 5.4]{BPa01}. As we shall make extensive use of this result, we recall it in detail in the next proposition. Note that part (a) of this proposition is contained in the proof of \cite[Corollary 5.4]{BPa01}, while part (b) is the original statement in \cite{BPa01}.

\begin{prop}[Corollary 5.4, \cite{BPa01}]\label{cor5.4}
Let $\mathrm{def}(\cQ(\FF)) = d \leq 1$ and let $\FF_0$ be a subfield of $\FF$.
\begin{enumerate}[{\rm (a)}]
\item\label{Co54-a}  Suppose that $\FF=\FF_0(\varepsilon)$ is a simple extension of $\FF_0$ and let $G$ be a generating set for $\cQ_2(\FF_0)$. Let $\ell_0$ and $\ell_1$ be two $\FF_0$-rational lines of $\cQ(\FF)$ such that $\ell_1^\perp\cap\ell_0 = \emptyset$ (i.e. $\ell_0$ and $\ell_1$ are opposite). Then $\cQ_2(\FF)$ is generated by any set of the form $G\cup\{t\}$ where $t=\langle p,q\rangle_{\cQ_2(\FF)}$, $p\in\ell_0$, $q=p^{\perp}\cap\ell_1$ and neither $p$ nor $q$ are $\KK$-rational, for any proper subfield $\KK$ of $\FF$ containing $\FF_0$.
\item\label{Co54-b} If the field $\FF$ is generated by adjoining $k$ elements to its subfield $\FF_0$, then the geometry $\cQ_2(\FF)$ can be generated by adding at most $k$ points to its subgeometry $\cQ_2(\FF_0)$.
\end{enumerate}
\end{prop}

\noindent
Since every finite field is a simple extension of its prime subfield, by part (\ref{Co54-b}) of Proposition \ref{cor5.4} and Cooperstein \cite{C98b} we immediately obtain the following:

\begin{prop}\label{gr k = 2, F finito}
Let $\FF$ be finite and $\mathrm{def}(\cQ(\FF)) = d \leq 1$. Then  ${{2n + d}\choose 2} \leq \gr(\cQ_2(\FF)) \leq {{2n+d}\choose 2}+1$  for any $n > 2$.
\end{prop}

The following characterization of when $\cQ_2(\FF)$ is generated over a subfield $\FF_0$ is also consequence of part (\ref{Co54-a}) of Proposition \ref{cor5.4}

\begin{lemma}\label{gset}
 Suppose $\FF$ is a simple field extension of $\FF_0.$ A subset $G \subseteq \cQ_2(\FF_0)$ generates $\cQ_2(\FF)$ if and only if $G$ generates $\cQ_2(\FF_0)$ and there exist two opposite lines $\ell,\ell'$ of $\cQ(\FF_0)$ and a line
  $m\in\langle G\rangle_{\cQ_2(\FF)}$ of $\cQ(\FF)$ such that each of the $\FF$-extensions $\langle \ell\rangle_\FF$ and $\langle \ell'\rangle_\FF$ of $\ell$ and $\ell'$ meets $m$ in a point and neither of the points $m\cap \langle \ell\rangle_\FF$ and $m\cap\langle \ell'\rangle_\FF$ is $\KK$-rational, for any proper subfield $\KK$ of $\FF$ containing $\FF_0$.
\end{lemma}

The next lemma follows from \cite[Theorem 1.3]{BPa01}; we provide a short proof here.

\begin{lemma}\label{l-gen}
Let $G$ be a generating set of $\cQ_2(\FF_0)$. Then the span $\langle G\rangle_{\cQ_2(\FF)}$ of $G$ in $\cQ_2(\FF)$ contains all lines of $\cQ(\FF)$ which meet the point-set of $\cQ(\FF_0)$ non-trivially.
\end{lemma}
\begin{proof}
Let $p$ be a point of $\cQ(\FF_0)$. Since $\langle G\rangle_{\cQ_2(\FF_0)} = \cQ_2(\FF_0)$, all lines of $\cQ(\FF_0)$
through $p$ belong to $\langle G\rangle_{\cQ_2(\FF)}$. On the other hand, the residue $\Res_\FF(p)$ of $p$ in $\cQ(\FF)$ is generated by the point-set of the residue $\Res_{\FF_0}(p)$ of $p$ in $\cQ(\FF_0)$, i. e.  the set of lines of $\cQ(\FF_0)$ through $p$. So, every line of $\cQ(\FF)$ through $p$ belong $\langle G\rangle_{\cQ_2(\FF)}$. The result follows.
\end{proof}

\begin{lemma} \label{r-pi}
Let $X$ be a plane of $\cQ(\FF)$ and $G$ a generating set for $\cQ_2(\FF_0)$. If $X$ contains at least two $\FF_0$-rational points then all lines of $X$ belong to $\langle G\rangle_{\cQ(\FF)}$.
\end{lemma}
\begin{proof}
Let $a$ and $b$ be two distinct $\FF_0$-rational points of $X$ and let $s$ be a line of $X$. Given a point $p\in s\setminus \{a,b\}$, let $\ell$ and $m$ be the lines joining $p$ with $a$ and $b$. Then $\ell, m\in \langle G\rangle_{\cQ_2(\FF)}$ by Lemma \ref{l-gen}. If $s\in \{\ell, m\}$ then we are done. Otherwise $\ell \neq m$ and $s\in \langle \ell, m\rangle_{\cQ_2(\FF)}$. However $\ell, m\in \langle G\rangle_{\cQ_2(\FF)}$. Therefore $s\in \langle G\rangle_{\cQ_2(\FF)}$.
\end{proof}

The next lemma is nothing but elementary linear algebra.

\begin{lemma}\label{triv}
Let $X$ and $Y$ be $\FF_0$-rational subspaces of $\PG(N-1,\FF)$. Then $X\cap Y$ too is $\FF_0$-rational. In particular, if $X\cap Y$ is a single point then that point belongs to $\PG(N-1,\FF_0)$.
\end{lemma}

Suppose that $\mathrm{def}(\cQ) = d \leq 1$. Modulo proportionality or equivalence, we can assume that $q$ admits the following expression, well defined over every subfield $\FF_0$ of $\FF$:
\begin{equation}\label{q-odd}
\left.  \begin{array}{ll}
    \displaystyle   q(x_1,\ldots,x_{2n}):=\sum_{i=1}^nx_{2i-1}x_{2i} & \mbox{ if $d=0$ } \\[.2cm]
    \displaystyle q(x_1,\ldots,x_{2n+1}):=\sum_{i=1}^nx_{2i-1}x_{2i}+x_{2n+1}^2 & \mbox{ if $d=1$}. \\
    \end{array}\right\}
\end{equation}
We shall also use the customary notation $Q(2n,\FF)$ for $\cQ(\FF)$ when $d=1$ ({\em parabolic} case) and
$Q^+(2n-1,\FF)$ when $d=0$ ({\em hyperbolic} case). Accordingly, $Q_k(2n,\FF)$ and $Q_k^+(2n-1,\FF)$ are the corresponding $k$-Grassmannians.

When $d=2$, with $\mathrm{def}_2(\cQ) = d_2 = d$ if $\ch(\FF) = 2$, then $q$ can always be given the following expression:
  \begin{equation}\label{q-ell}
    q(x_1,\ldots, x_{2n+2}):=\sum_{i=1}^nx_{2i-1}x_{2i}+
    (x_{2n+1}^2+\lambda x_{2n+1}x_{2n+2}+\mu x_{2n+2}^2),
  \end{equation}
where the polynomial $t^2+\lambda t+\mu$ is irreducible over $\FF$. This is the so-called {\em elliptic} case; the customary symbol for $\cQ(\FF)$ in this case is $Q^-(2n+1,\FF)$. We warn that in general, given a subfield $\FF_0$ of $\FF$, we cannot exploit equivalence and proportionality in such a way as to obtain a polynomial over $\FF_0$ at the right side of \eqref{q-ell}, but in a few cases we can.

The next theorem deals with the hyperbolic case. It corresponds to Theorem \ref{MT3} of the Introduction.

\begin{theorem}\label{not-gen}
The line-Grassmannian $Q_2^+(5,\FF)$ of $Q^+(5,\FF)$ is never $\FF_0$--generated, for any proper subfield $\FF_0$ of $\FF.$
\end{theorem}
   \begin{proof}
Suppose $\FF_0<\FF$ and take $\varepsilon\in\FF\setminus\FF_0$. Let $\Omega_1$ be the set of lines of $Q^+(5,\FF)$ which meet $Q^+(5,\FF_0)$ non-trivially and $\Omega_2$ the set of lines of $Q^+(5,\FF)$ contained in $\FF_0$-rational planes of $Q^+(5,\FF)$. Put $\Omega := \Omega_1\cup \Omega_2$. We firstly prove the following:

\begin{itemize}
\item[(a)] The set $\Omega$ is a proper subset of $Q^+_2(5,\FF)$.
\end{itemize}
Consider the following line of $Q^+(5,\FF)$:
\[ \ell := \langle (1,0,\varepsilon,0,0,0),     (0,\varepsilon,0,-1,0,0)\rangle_V\]
where $V = V(6,\FF)$ and we take the liberty to identify subspaces of $V$ with the corresponding subspaces of $\PG(V)$. Then $\ell^{\perp} = X_1\cup X_2$ where $X_1$ and $X_2$ are the following planes of $Q^+(5,\FF)$
\[X_1:= \langle\ell, (0,0,0,0,1,0)\rangle_V, \hspace{5 mm}  X_2:= \langle \ell,(0,0,0,0,0,1)\rangle_V. \]
The line $\ell$ belongs to neither $\Omega_1$ nor $\Omega_2$. Indeed $\ell$ has no $\FF_0$-rational point and $X_1$ and $X_2$ are the only two planes of $Q^+(5,\FF)$ which contain $\ell$, but each of them admits exactly one $\FF_0$-rational point. Therefore $\ell \not\in \Omega$. Claim (a) is proved. Consider now the following claim:

\begin{itemize}
\item[(b)] the set $\Omega$ is a subspace of $Q^+_2(5,\FF)$.
\end{itemize}
If (b) holds true then $\Omega = \langle Q^+(5,\FF_0)\rangle_{Q^+_2(5,\FF)}$ by  Lemmas \ref{l-gen} and \ref{r-pi}. Hence $\langle Q^+(5,\FF_0)\rangle_{Q^+_2(5,\FF)} \subset Q^+_2(5,\FF)$ by (a) and the theorem is proved. Claim (b) remains to be proved. We shall firstly prove the following
     \begin{itemize}
     \item[(c)] Every plane of $Q^+(5,\FF)$ containing two distinct $\FF_0$-rational points is $\FF_0$-rational.
     \end{itemize}
Let $X$ be a plane of $Q^+(5,\FF)$ and let $a, b$ be two $\FF_0$-rational points of $X$. The line $r$ of $Q^+(5,\FF_0)$ spanned by $a$ and $b$
belongs to just two planes $X_1$ and $X_2$ of $Q^+(5,\FF_0)$. Their $\FF$-extensions $\langle X_1\rangle_\FF$ and $\langle X_2\rangle_\FF$ are the two planes of $Q^+(5,\FF)$ through the $\FF$-extension $\langle r\rangle_\FF$ of $r$. The plane $X$ is one of them. Hence $X$ is $\FF_0$-rational, as claimed in (c).

We are now ready to prove (b). Let $x$ and $y$ be two distinct elements of $\Omega$, collinear as points of $Q^+_2(5,\FF)$, and let $L$ be the line of $Q^+_2(5,\FF)$ spanned by them. We must show that $L\subseteq \Omega$. Three cases must be considered.

\medskip

\noindent
(1) $x, y \in \Omega_1$, namely each of $x$ and $y$ contains an $\FF_0$-rational point. Put $p = x\cap y$ be the meet-point of the lines $x$ and $y$ (which exists since $x$ and $y$ are collinear in $Q^+_2(5,\FF)$. If $p$ is $\FF_0$-rational, then all elements of $L$ belong to $\Omega_1$. Suppose that $p$ is not $\FF_0$-rational. The plane $X$ spanned by $x$ and $y$ (which is a plane of $Q^+(5,\FF)$) contains at least two $\FF_0$-rational points, one of which is contributed by $x$ and the other one by $y$. Hence $X$ is $\FF_0$-rational by claim (c). Therefore $L\subseteq \Omega_2$. So, in either case $L\subseteq \Omega$.

\medskip

\noindent
(2)  $x, y\in\Omega_2.$ Let $X$ and $Y$ be $\FF_0$-rational planes of $Q^*(5,\FF)$ containing $x$ and $y$ respectively. Let also $Z$ be the plane of $Q^+(5,\FF)$ spanned by $x$ and $y$ and $p := x\cap y$. If $Z$ is $\FF_0$-rational then $L\subseteq \Omega_2$. So, suppose that $Z$ is not $\FF_0$-rational. Then $X, Y, Z$ are pairwise distinct. In $\Res(p)$ they appear as three distinct lines, with $Z$ meeting each of $X$ and $Y$ in a point (namely $x$ and $y$ respectively).  So, $X$ and $Y$ belong to the same family of lines of the grid $\Res(p)$. Consequently, as planes of $Q^+(5,\FF)$, they meet in a single point, namely $p$. However $X$ and $Y$ are $\FF_0$-rational. Therefore $p$ is $\FF_0$-rational, by Lemma \ref{triv}. Hence $L\subseteq \Omega_1$.

\medskip

\noindent
(3) One of the lines $x$ and $y$ belongs to $\Omega_1\setminus \Omega_2$ and the other one belongs to $\Omega_2\setminus \Omega_1$. We shall see that this situation leads to a contradiction, thus finishing the proof of (b).

Let $x\in\Omega_1\setminus\Omega_2$ and $y\in\Omega_2\setminus\Omega_1$, to fix ideas.
Put $p = x\cap y$ and let $X$ be the plane of $Q^+(5,\FF)$ spanned by $x$ and $y$. Then neither $p$ nor $X$ are $\FF_0$-rational, since $y\not\in \Omega_1$ and $x\not\in \Omega_2$. On the other hand, $x\in \Omega_1$ contains at least one $\FF_0$-rational point. Let $q$ be one of them. Moreover, if $Y$ is the plane of $Q^+(5,\FF)$ through $y$ different from $X$, then $Y$ is $\FF_0$-rational, since $y\in \Omega_2$. As $q$ and $Y$ are $\FF_0$-rational,  $Q^+(5,\FF_0)$ contains $q$ and the $\FF_0$-section $Y_0$ of $Y$. Consider the line $y_0 = q^\perp\cap Y_0$ of $Q^+(5,\FF_0)$. The $\FF$-extension $\langle y_0\rangle_\FF$ of $y_0$ is contained in $Y = \langle Y_0\rangle_\FF$ and is orthogonal to $q$. However $q^\perp\cap Y = y$ in $Q^+(5,\FF)$. Therefore $y = \langle y_0\rangle_\FF$. So, $y$ is $\FF_0$-rational. We have reached a  contradiction.
   \end{proof}

\begin{remark}\label{not-gen-c}
The proof of Theorem \ref{not-gen} exploits the fact that every line of $Q^+(5,\FF)$ is contained in precisely two singular planes. This suggests the conjecture that  $Q^+_{n-1}(2n-1,\FF)$ is never $\FF_0$--generated, for any $n > 2$ and any proper subfield $\FF_0$ of $\FF$.
\end{remark}

\begin{note}
Theorem \ref{not-gen} implies that, for a prime $p$, no generating set of $Q_2^+(5,p)$ generates $Q_2^+(5,p^h)$ for $h>1$. However, by Proposition \ref{cor5.4}, if $G_0$ is a generating set for $Q_2^+(5,p)$ then there exist a line $\ell$ of $Q^+(5,p^h)$ such that $G_0\cup\{\ell\}$ generates $Q_2^+(5,p^h)$. So, the generating rank of $Q^+_2(5,p^h)$ is at most $16$ (recall that $15$ is the generating rank of $Q_2^+(5,p)$, by \cite{C98b}). This does not preclude the possibility that $Q_2^+(5,p^h)$ admits a generating set of size 15 but such a set, if it exists, cannot be contained in $Q^+_2(5,p^r)$, for any $r < h$.
\end{note}

We shall shall now turn to the proof of Theorem \ref{MT4} of the Introduction, but we firstly state a preliminary technical lemma.

Let $\cQ = Q(6,\FF)$. According to \eqref{q-odd}, we can assume that $\cQ$ is associated to the following quadratic form $q:V(7,\FF)\rightarrow\FF$:
\[ q(x_1,\dots,x_7) = x_1x_2+x_3x_4+x_5x_6+x_7^2. \]
The following is the bilinearization of $q$:
\[ \begin{array}{rcl}
f((x_1,\dots,x_7),(y_1,\dots,y_7)) & = &  \sum_{i=1}^3(x_{2i-1}y_{2i}+x_{2i}y_{2i-1})+2x_7y_7 \\
 & & (= \sum_{i=1}^3(x_{2i-1}y_{2i}+x_{2i}y_{2i-1}) \hspace{3 mm}\mbox{when}~ \ch(\FF) = 2).
\end{array} \]
Given a proper subfield $\FF_0$ of $\FF$, let $\varepsilon\in\FF\setminus\FF_0$ be such that $\FF=\FF_0(\varepsilon)$. Let $\ell_\varepsilon$ be the line of $\PG(5,\FF)$ corresponding to the following subspace of $V = V(7,\FF)$:
\[\langle (0,\varepsilon,\varepsilon^{-1},-\varepsilon,0,-1,1),
       (1,1,0,0,\varepsilon,-\varepsilon^{-1},0)\rangle_V.\]
Comparing the above expressions for $q$ and $f$, it is straightforward to check that $\ell_\varepsilon$ is totally singular for $q$, namely it is a line of $\cQ$

\begin{lemma}\label{m-gen}
With $\cQ$, $\FF_0$, $\varepsilon$ and $\ell_\varepsilon$ as above, let $\cQ_2 = Q_2(6,\FF)$ be the line-Grassmannian of $\cQ$ and suppose that  $\ell_\varepsilon$ is $\FF_0$-generated. Then $\cQ_2$ is $\FF_0$--generated.
\end{lemma}
\begin{proof}
Henceforth, given a non-zero vector $(x_1,..., x_7)$ of $V$ we denote the corresponding point of $\PG(V)$ by the symbol $[x_1,..., x_7]$. For short, we denote spans in $\PG(V)$ by the symbol $\langle ... \rangle$ instead of $\langle ...\rangle_{\PG(V)}$.

Consider the projective line $t:=\langle p_1, p_2\rangle$ spanned by the points $p_1 = [1,0,0,0,\varepsilon,0,0]$ and $p_2 = [0,\varepsilon,0,0,0,-1,0]$. Both $p_1$ and $p_2$ belong to $\cQ$ and $t$ is a line of $\cQ$.
Let $t_1$ and $t_2$ be the following lines of $\cQ$:
\[t_1 := \langle [1,0,0,0,0,0,0],[0,0,0,0,1,0,0] \rangle, \hspace{5 mm} t_2 := \langle [0,1,0,0,0,0,0],[0,0,0,0,0,1,0] \rangle.\]
The lines $t_1$ and $t_2$ are opposite (i.e. $t_i^\perp\cap t_j = \emptyset$ for $\{i,j\} = \{1,2\}$) and $p_i \in t_i$ for $i = 1, 2$. Moreover, neither $p_1$ nor $p_2$ are defined over any proper subfield of $\FF$ containing $\FF_0$, since $\FF_0(\varepsilon) = \FF$. So, $t, t_1, t_2, p_1, p_2$ are as in the setting of Lemma \ref{gset}. By that lemma, if $t$ is $\FF_0$-generated then $\cQ$ is also $\FF_0$-generated. Hence in the remainder of the proof we will be concerned in proving that $t$ is $\FF_0$-generated. Put
   \[s := \langle [0,0,\varepsilon^{-1},-\varepsilon,0,0,1], [1,1,0,0,\varepsilon,-\varepsilon^{-1},0]\rangle.\]
Then $s$ is a line of $\cQ$ and is collinear with $\ell_\varepsilon$ in $\cQ_2$. Moreover, $t\in \langle s, \ell_\varepsilon\rangle_{\cQ_2}$.
The line $\ell_\varepsilon$ is $\FF_0$-generated. Hence it is enough to prove that $s$ is also $\FF_0$--generated. Consider
\[s_{1} := \langle [0,0,\varepsilon^{-1},-\varepsilon,0,0,1], [1,0,0,0,\varepsilon,0,0]\rangle, ~~ s_{2} := \langle [0,0,\varepsilon^{-1},-\varepsilon,0,0,1], [0,1,0,0,0,-\varepsilon^{-1},0]\rangle.\]
Then $s\in\langle s_1, s_2\rangle_{\cQ_2}$ Also, $s_1\in \langle s_{11}, s_{12}\rangle_{\cQ_2}$ and  $s_2\in\langle s_{21}, s_{22}\rangle_{\cQ_2}$ where
\[s_{11}  := \langle[0,0,\varepsilon^{-1},-\varepsilon,\varepsilon^{2},0,1], [1,0,0,0,0,0,0]\rangle, ~~~ s_{12} := \langle [\varepsilon,0,\varepsilon^{-1},-\varepsilon,0,0,1], [0,0,0,0,1,0,0]\rangle,\]
\[s_{21} := \langle[0,0,\varepsilon^{-1},-\varepsilon,0,-\varepsilon^{-1},1],  [0,1,0,0,0,0,0]\rangle, ~ s_{22} := \langle  [0,1,\varepsilon^{-1},-\varepsilon,0,0,1],  [0,0,0,0,0,1,0]\rangle.\]
By Lemma~\ref{l-gen} the lines $s_{11}$, $s_{12}$, $s_{21}$ and $s_{22}$ are
     $\FF_0$--generated; it follows that $s_1$ and $s_2$ are $\FF_0$--generated and so also  $s$ is  $\FF_0$--generated.
     The proof is complete.
        \end{proof}

The next theorem yields the core of Theorem \ref{MT4} of the Introduction.

\begin{theorem}\label{t-gen}
If $\FF$ is $\FF_4, \FF_8$ or $\FF_9$ then $Q_2(6,\FF)$ is generated over the prime subfield of $\FF$.
\end{theorem}
\begin{proof}
Let $\FF_0$ be the prime subfield of $\FF$. Since we are in the hypothesis of Lemma \ref{m-gen} it is enough to prove that, in each of the three cases considered in the theorem, the line $\ell_\varepsilon$ of Lemma \ref{m-gen} is $\FF_0$-rational;  then apply Lemma \ref{m-gen}.

\medskip

\noindent
Let $\FF = \FF_4$. We have $\FF_4=\FF_2[\varepsilon]$ for $\varepsilon\in \FF_4\setminus\FF_2$ such that $\varepsilon^2+\varepsilon+1=0$. Take
\[\ell_{1} := \langle[1,0,0,1,1,1,1], [1,\varepsilon^2,\varepsilon,\varepsilon,0,1,0]\rangle, ~~~ \ell_{2} := \langle [1,1,1,0,0,0,1], [1,0,1,1,\varepsilon^2,\varepsilon,0]\rangle.\]
Then $\ell_1, \ell_2$ are lines of $Q(6,4)$ and $\ell_\varepsilon\in \langle \ell_1, \ell_2\rangle_{\cQ_2}$, where $\cQ_2 = Q_2(6,4)$. Since both $\ell_1$ and $\ell_2$ contain points defined over $\FF_2$, by Lemma~\ref{l-gen} they are both $\FF_2$-generated. Hence the line $\ell_\varepsilon$ is $\FF_2$--generated.

\medskip

\noindent
Let $\FF = \FF_8$. Now $\FF_8=\FF_2[\varepsilon]$ with $\varepsilon^3+\varepsilon+1=0$. Take
\[\ell_{1} := \langle [0,0,1,\varepsilon^4,1,\varepsilon^5,1], [1,\varepsilon,\varepsilon^4,\varepsilon^4,\varepsilon^4,0,0]\rangle,~~ \ell_{2} :=  \langle [\varepsilon^5,0,\varepsilon^2,\varepsilon^3,\varepsilon,\varepsilon^3,1], [0,1,\varepsilon,\varepsilon,\varepsilon^6,\varepsilon^3,0]\rangle. \]
Then $\ell_\varepsilon\in\langle \ell_1, \ell_2\rangle_{\cQ_2}$ with  $\ell_1, \ell_2\in Q(6,8)$. Put
\[\ell_{11} := \langle [0,0,1,1,1,0,1], [1,\varepsilon,\varepsilon^4,\varepsilon^4,\varepsilon^4,0,0]\rangle, ~~ \ell_{12} := \langle [0,0,0,1,0,1,0],  [1,\varepsilon,\varepsilon^4,\varepsilon^4,\varepsilon^4,0,0]\rangle,\]
\[\ell_{21} := \langle[1,0,1,0,1,0,0], [0,1,\varepsilon,\varepsilon,\varepsilon^6,\varepsilon^3,0]\rangle, ~~~~ \ell_{22} := [0,1,1,1,0,0,1],
 [0,1,\varepsilon,\varepsilon,\varepsilon^6,\varepsilon^3,0]\rangle.\]
Then  $\ell_1\in \langle \ell_{11}, \ell_{12}\rangle_{\cQ_2}$ and  $\ell_2\in \langle\ell_{21}, \ell_{22}\rangle_{\cQ_2}.$
 By Lemma~\ref{l-gen} the lines $\ell_{11},\ell_{12},\ell_{21}$ and $\ell_{22}$ are $\FF_2$-generated; so $\ell_\varepsilon$ too is $\FF_2$-generated.

\medskip

\noindent
Let $\FF = \FF_9$. We have $\FF_9=\FF_3[\varepsilon]$ with $\varepsilon^2-\varepsilon-1=0$. Put
\[\ell_1 := \langle [0,\varepsilon^7,\varepsilon^7,\varepsilon^7,\varepsilon^5,-1,1],  [1,1,0,0,\varepsilon,\varepsilon^3,0]\rangle, ~~ \ell_2 :=
\langle [0,\varepsilon^5,\varepsilon^7,\varepsilon^2,\varepsilon^2,-1,1], [1,1,0,0,\varepsilon,\varepsilon^3,0]\rangle. \]
Then $\ell_\varepsilon\in \langle\ell_1,\ell_2\rangle_{\cQ_2}$ with  $\ell_1, \ell_2\in \cQ_2(6,9).$ Also,  $\ell_1\in \langle\ell_{11},\ell_{12}\rangle_{\cQ_2}$ and $\ell_2\in\langle \ell_{21},\ell_{22}\rangle_{\cQ_2}$ where
\[\ell_{11} := \langle [0,\varepsilon^3,\varepsilon^3,\varepsilon^3,\varepsilon^6,\varepsilon^3,1], [1,-1,1,1,0,0,0]\rangle, ~ \ell_{12} :=
\langle [0,-1,-1,-1,1,1,1], [1,\varepsilon^5,\varepsilon^6,\varepsilon^6,1,\varepsilon^2,0]\rangle,\]
\[\ell_{21} := \langle  [1,-1,1,0,0,0,1], [1,\varepsilon^5,\varepsilon^2,\varepsilon^2,1,\varepsilon^2,0]\rangle, ~~ \ell_{22} := \langle  [-1,1,0,-1,0,0,1], [0,1,-1,-1,\varepsilon^5,\varepsilon^7,0] \rangle.\]
 By Lemma~\ref{l-gen} the lines $\ell_{11},\ell_{12},\ell_{21}$ and $\ell_{22}$ are $\FF_3$-generated; hence $\ell_\varepsilon$ is $\FF_3$-generated as well.
 \end{proof}

\begin{note}
Most likely the statement of Theorem \ref{t-gen} holds for infinitely many finite fields, perhaps for all of them, but the proof of Theorem \ref{t-gen} contains no obvious hints for generalizations. In particular, it is not clear if a general rule is implicit in the way the lines $\ell_i$ and $\ell_{ij}$ are chosen.
\end{note}

In order to finish the proof of Theorem \ref{MT4} we need the following general lemma.

\begin{lemma}\label{l-bound}
Given any field $\FF$, let $\cQ = \cQ(\FF)$ with rank $n > 2$ and defect $d \leq 2$. When $d = 2$ and $\ch(\FF) = 2$ assume furthermore that $\mathrm{def}_2(\cQ) = d$. Then $\gr(\cQ_2)\geq{{2n+d}\choose 2}$.
\end{lemma}
\begin{proof}
Put $N := 2n+d$. When $\ch(\FF)\neq 2$ the Pl\"{u}cker embedding of $\cQ_2$ has dimension $N\choose 2$; see \cite{ILP}. So $\gr(\cQ_2)\geq\er(\cQ_2)\geq{N\choose 2}$. Let $\ch(\FF)=2$. If $d \leq 1$ then $\cQ_2$ admits the so-called Weyl embedding (see \cite{IP13}), which is ${N\choose 2}$-dimensional. Hence $\gr(\cQ_2)\geq{N\choose 2}$ in this case too. Finally, let $\ch(\FF) = 2$ and $d = 2$. Then $(\mathrm{def}_1(\cQ), \mathrm{def}_2(\cQ)) = (0,2)$ by assumption. In this case there exist a non-degenerate orthogonal polar space $\cQ' = \cQ'(\FF)$ of rank $n+1$ and defect $\mathrm{def}(\cQ') = 1$ such that $\cQ \cong H$ for a hyperplane $H$ of $\cQ'$. By Proposition \ref{lemma1} the line-Grassmannian $\cQ'_2$ admits a generating set of the form $S_k(H, p_0, \ell_0)$. By the above,
\[\gr(\cQ'_2) ~ \geq ~ {{2(n+1)+\mathrm{def}(\cQ')}\choose 2} ~ = ~ {{2n+3}\choose 2} ~ = ~ {{(2n+d)+1}\choose 2} ~ = ~{{N+1}\choose 2}.\]
Hence $\gr(H_2) + \gr(\Res(p_0)) + 1 \geq {{N+1}\choose 2}$. However $\gr(H_2) = \gr(\cQ_2)$ since $H \cong \cQ$ and the residue $\Res(p_0)$ of $p_0$ in $\cQ'$ is generated by $2n+1 = N-1$ of its points (Corollary \ref{t:k=1 symplectic even}). Therefore
$\gr(\cQ_2) + N-1 + 1 \geq {{N+1}\choose 2}$, namely $\gr(\cQ_2) \geq {N\choose 2}$.
\end{proof}

The next corollary finishes the proof of Theorem \ref{MT4} of the Introduction.

\begin{corollary}\label{c-orth}
For $q\in\{4,8,9\}$ let $\cQ(\FF_q)$ be a non-degenerate orthogonal polar space defined over $\FF_q$ with reduced Witt index $n \geq 3$ and defect $d \leq 2$, with $d > 0$ when $n = 3$. Then $\gr(\cQ_2(\FF_q)) = {{2n+d}\choose 2}$. If moreover $d\leq 1$ then $\cQ_2(\FF_q)$ is generated over the prime subfield of $\FF_q$.
\end{corollary}
\begin{proof}
Throughout this proof $\FF = \FF_q$ with $q$ as above, $\FF_0$ is the prime subfield of $\FF$ and $\cQ = \cQ(\FF)$ with $n$ and $d$ as above. Note that, as $\FF$ is finite, when $d = 2$ and $\ch(\FF) = 2$ then necessarily $\mathrm{def}_2(\cQ) = 2$. So, the hypotheses of Lemma \ref{l-bound} are satisfied. To begin with, we show that

\begin{itemize}
\item[$(\ast)$] the $2$-Grassmannian $\cQ_2$ admits a generating set $G$ of cardinality at most
${{2n+d}\choose 2}$. Moreover, if $d \leq 1$ then $G$ can be chosen to be $\FF_0$-rational.
\end{itemize}
We proceed by induction on $n$, firstly assuming that $d \leq 1$.

\medskip

\noindent
Step 1. When $n=3$ and $d=1$ claim $(\ast)$ follows from Theorem~\ref{t-gen} and \cite{C98b}.

\medskip

\noindent
Step 2. Let  $n > 3$ and  $d=0$. We can assume that the quadratic form associated to $\cQ$ is given as in case $d = 0$ of \eqref{q-odd}. Let $X$ be the hyperplane of $\PG(2n-1,\FF)$ represented by the equation $x_{2n-1}=x_{2n}$. Then $H:=\cQ\cap X\cong Q(2n-2,\FF)$ is a hyperplane of $\cQ$. Take a point $p_0\in\cQ\setminus H$ defined over $\FF_0$, e.g. $p_0= [0,..., 0,1]$. By Proposition~\ref{lemma1} the geometry $\cQ_2$ admits a generating set $S_2(H, p_0, \ell_0) = H_2\cup S_2(p_0)\cup\{\ell_0\}$ where $\ell_0\in \cQ_2$ is such that $\ell_0\not\subseteq H\cup p_0^{\perp}$ and $H\cap\ell_0\neq p_0^{\perp}\cap\ell_0$. It is not difficult to see that we can also choose $\ell_0$ in such a way that $\ell_0$ is $\FF_0$-rational. Claim $(\ast)$ holds for $H_2$ either by step 1 or by the inductive hypothesis, since $H$ is a non degenerate polar space of rank $n-1$ and defect $1$ and it is $\FF_0$-rational. Hence $H_2$ admits an $\FF_0$-rational generating set of size at most ${{2n-1}\choose 2}$. The set $S_2(p_0)$ is the point-set of $\Res(p_0)\cong Q^+(2n-3,\FF)$. Hence $2n-2$ lines through $p_0$ are enough to generate it. As $p_0$ is $\FF_0$-rational, we can choose those $2n-2$ lines in such a way that all of them are $\FF_0$-rational. Tu sum up, we can generate $\cQ_2$ by at most ${{2n-1}\choose 2}+ (2n-2) + 1 = {{2n}\choose 2}$ elements, all of which can be chosen to be $\FF_0$-rational, as claimed in $(\ast)$.

\medskip

\noindent
Step 3. Let $n > 3$ and  $d= 1$. We can assume that the quadratic form associated to $\cQ$ is given as in case $d = 1$ of \eqref{q-odd}. Let $X$ be the hyperplane of $\PG(2n,\FF)$ represented by the equation $x_{2n+1} = 0$. Then $H:=\cQ\cap X\cong Q(2n-1,\FF)$ is a hyperplane of $\cQ$.
We obtain the conclusion as in step 2, except that now when dealing with $H_2$ we refer to step 2 instead of the inductive hypothesis.

\medskip

\noindent
Step 4. Finally, let $d = 2$. Recalling that $\mathrm{def}_2(\cQ) =  2$ when $\FF$ is $\FF_4$ or $\FF_8$, the quadratic form defining $\cQ$ can be taken as in \eqref{q-ell}. Let $X$ be the hyperplane of $\PG(2n+1,\FF)$ of equation $x_{2n+2}=0$. Then $H:=\cQ\cap X\cong Q(2n,\FF)$ is a hyperplane of $\cQ$. So, we can apply the same argument as in steps 2 and 3, but referring to step 3 when dealing with $H_2$ and without caring  of choosing an $\FF_0$-rational line as $\ell_0$ (although we can) and an $\FF_0$ rational generating set for $S_2(p_0)$ (which we do not know if it exists).

The proof of $(\ast)$ is complete. The statement of the corollary now follows by comparing $(\ast)$ with Lemma \ref{l-bound}.
\end{proof}

\begin{note}
The restriction $d \leq 1$ in the last claim of Corollary \ref{c-orth} is due to the fact that when $d = 2$ the scalars $\lambda$ and $\mu$ in \eqref{q-ell} need not belong to $\FF_0$. If there is no way to choose them in $\FF_0$ by replacing the given quadratic form with another one equivalent or proportional to it, then $\cQ$ itself is not $\FF_0$-generated.
\end{note}

\vskip.2cm\noindent
\begin{minipage}[t]{\textwidth}
Authors' addresses:
\vskip.2cm\noindent\nobreak
\centerline{
\begin{minipage}[t]{7cm}
Ilaria Cardinali and Antonio Pasini \\
Department of Information Engineering and Mathematics\\
University of Siena\\
Via Roma 56, I-53100, Siena, Italy\\
ilaria.cardinali@unisi.it\\
antonio.pasini@unisi.it
\end{minipage}\hfill
\begin{minipage}[t]{6cm}
Luca Giuzzi\\
D.I.C.A.T.A.M. \\ Section of Mathematics \\
Universit\`a di Brescia\\
Via Branze 53, I-25123, Brescia, Italy \\
luca.giuzzi@unibs.it
\end{minipage}}
\end{minipage}

\bigskip


 \end{document}